\documentclass[12pt]{amsart}
\usepackage{amsmath,amsthm}
\usepackage{latexsym}
\usepackage{amssymb}
\usepackage{amsfonts}
\usepackage{amscd}
\usepackage{booktabs}
\usepackage{slashbox}
\usepackage{dsfont} 
\usepackage{mathtools} 
\usepackage{braket} 
\input{xy}
\xyoption{all}

\newcounter{eqcounter}[section]
\renewcommand{\theeqcounter}{\arabic{section}.\arabic{eqcounter}}
\renewenvironment{equation}{\medskip\noindent\refstepcounter{eqcounter}\makebox[0pt][l]{({\bf\theeqcounter})}\begin{minipage}[b]{\textwidth}$$}{$$\end{minipage}\medskip\noindent}

\newcommand{\ch}{\mathrm{ch}}
\newtheorem{thm}{Theorem}[section]
\newtheorem{theorem}[thm]{Theorem}
\newtheorem{prop}[thm]{Proposition}

\newtheorem{corollary}[thm]{Corollary}
\newtheorem{lemma}[thm]{Lemma}

\newtheorem{proposition}[thm]{Proposition}

\theoremstyle{definition}

\newtheorem{definition}[thm]{Definition}

\newtheorem{remark}[thm]{Remark}

\newtheorem{definition-proposition}[thm]{Definition-Proposition}
\DeclareMathOperator{\End}{End}

\newcommand{\R} {\mathbb{R}}
\newcommand{\Z} {\mathbb{Z}}
\newcommand{\bZ} {\mathbf{Z}}
\newcommand{\C} {\mathbb{C}}

\newcommand{\br}{\mathbf r}
\newcommand{\zbar}{\overline{z}}

\newcommand{\cA}{\mathcal{A}}

\newcommand{\cD}{\mathcal{D}}
\newcommand{\DD}{\mathcal{D}}

\newcommand{\cI}{\mathcal {I}}
\newcommand{\cJ}{{\mathcal J}}
\newcommand{\cM}{{\mathcal M}}
\newcommand{\cO}{\mathcal{O}}
\newcommand{\cP}{\mathcal{P}}
\newcommand{\cS}{{\mathcal S}}

\newcommand{\cU}{\mathcal{U}}
\newcommand{\cL}{\mathcal{L}}
\newcommand{{\cHol}}{{\mathcal{{H}}}}
\newcommand{\cV}{\mathcal V}
\newcommand{{\cLf}}{\mathcal{L}_{\mathrm{f}}}
\newcommand{{\cLfo}}{\wt{\cLf}} 

\newcommand{\al}{\alpha}
\newcommand{\be}{\beta}
\newcommand{\ga}{\gamma}
\newcommand{\de}{\delta}
\newcommand{\ep}{\epsilon}

\newcommand{\gk}{\frak g}
\newcommand{\hk}{\frak h}

\newcommand{\fg}{\mathfrak g}
\newcommand{\fc}{\mathfrak c}

\newcommand{\fn}{\mathfrak n}
\newcommand{\fh}{\mathfrak h}
\newcommand{\fk}{\frak k}
\newcommand{\fsu}{\mathfrak{su}}

\newcommand{\fb}{\mathfrak b}

\newcommand{\fso}{\mathfrak{so}}
\newcommand{\fsp}{\mathfrak{sp}}

\newcommand{\fm}{\mathfrak{m}}

\newcommand{\fgl}{\mathfrak{gl}}
\newcommand{\fsl}{\mathfrak{sl}}
\newcommand{\fp}{\mathfrak p}
\newcommand{\fr}{\mathfrak r}
\newcommand{\fu}{\mathfrak u}
\newcommand{\fq}{\mathfrak q}

\newcommand{\smflds}{ \mathrm{(smflds) }}

\newcommand{\rGL}{\mathrm{GL}}

\newcommand{\rOsp}{{\hbox{Osp}}}
\newcommand{\rosp}{{\mathrm{osp}}}

\newcommand{\rAut}{{\mathrm{Aut}}}

\newcommand{\rsl}{{\mathrm{sl}}}

\newcommand{\rM}{\mathrm{M}}
\newcommand{\rSp}{\mathrm{Sp}}

\newcommand{\Ad}{\mathrm{Ad}}
\newcommand{\ad}{\mathrm{ad}}
\newcommand{\Lie}{\mathrm{Lie}}

\newcommand{\rk}{{\mathrm{rk}}}

\newcommand{\ddiag}{\mathrm{diag}}

\newcommand{\uHom}{\mathrm{\underline{Hom}}}

\newcommand{\albar}{\overline{\alpha}}
\newcommand{\bebar}{\overline{\beta}}

\newcommand{\cinfty}{C^\infty} 
\newcommand{\red}[1]{\widetilde{ #1 }}
\newcommand{\restr}[2]{{{#1}}_{{\left. \right|}_{#2}}}

\usepackage{color}

\renewcommand{\to}{\mathchoice{\longrightarrow}{\rightarrow}{\rightarrow}{\rightarrow}} 
\let\mto\mapsto \renewcommand{\mapsto}{\mathchoice{\longmapsto}{\mto}{\mto}{\mto}} 

\newcommand{\wt}{\widetilde}

\newcommand{\di}{{\rm d}}

\newcommand{\id}{{\mathrm{id}}}

\newcommand{\Pol}{\textrm{Pol}}

\newcommand{\lra} {\longrightarrow}

\newcommand{\cov}[1]{{#1}^{\sim}}

\newcommand{\beq}{\begin{equation}}
\newcommand{\eeq}{\end{equation}}

\begin{document}

\Large
\centerline{\bf Highest weight Harish-Chandra supermodules} 

\medskip

\centerline{\bf and their geometric realizations}


\normalsize


\bigskip

\centerline{C. Carmeli$^\natural$, R. Fioresi$^\flat$
\footnote{The research was supported by 
MSCA-RISE EU project GHAIA grant number 777822}, V. S. Varadarajan $^\star$ }

\bigskip
\centerline{\it $^\natural$ DIME, 
Universit\`a di Genova }
 \centerline{\it Via Armando Magliotto, 2, 17100 Savona, Italy}
\centerline{{\footnotesize e-mail: carmeli@dime.unige.it}}

\medskip
\centerline{\it  $^\flat$ Dipartimento di Matematica, Universit\`{a} di
Bologna }
 \centerline{\it Piazza di Porta S. Donato, 5. 40126 Bologna, Italy.}
\centerline{{\footnotesize e-mail: rita.fioresi@unibo.it}}

\medskip
\centerline{\it $^\star$
University of California, Los Angeles}
\centerline{\it Los Angeles, CA 90095-1555, USA}
\centerline{{\footnotesize e-mail: vsv@math.ucla.edu}}

\begin{abstract}
In this paper we discuss the highest weight 
$\fk_r$-finite representations
of the pair $(\fg_r,\fk_r)$ consisting of $\fg_r$, a real
form of a complex basic Lie superalgebra of classical type
$\fg$ ($\fg\neq A(n,n)$), 
and the maximal compact subalgebra $\fk_r$ of $\fg_{r,0}$,
together with their geometric global realizations.
These representations occur, as in the ordinary setting, in the
superspaces of sections of holomorphic super vector bundles on
the associated Hermitian superspaces $G_r/K_r$. 
\end{abstract}

\section{Introduction}

During 1955-56 Harish-Chandra published three papers in the American Journal 
of Mathematics devoted to understanding the
theory of representations of a real semisimple Lie
group, which are also (infinitesimally) highest weight modules \cite{hc}. 
These modules were constructed both infinitesimally and globally, 
the global modules realized as spaces of sections of certain holomorphic 
vector bundles on the associated symmetric space, which is hermitian symmetric. 
He constructed the matrix element defined by the highest weight vector, 
and verified its square integrability under suitable conditions. 
Under these conditions, the representations were in the discrete series and 
he obtained formulae for their formal degree and character, which showed 
a strong resemblance to the Weyl formulae in the finite dimensional case. 
These calculations convinced him of the structure of the discrete series 
in the general case, although he was still years away from 
resolving this puzzle completely.

\medskip

The purpose of this paper
is to generalize some aspects of this theory to the
supersymmetric situation.

\medskip
{In the literature, several authors have discussed unitarizable
infinitesimal Harish-Chandra modules, 
see for example \cite{jakobsen, dz, fn} and Refs. within.
In \cite{jakobsen} Jakobsen classifies the unitarizable highest
weight modules for Lie superalgebras, while Furutsu and Nishiyama 
in \cite{fn} focus
on the case $\fsu(p,q|n)$. 
A global realization of such modules for the conformal supergroup
appears in \cite{dp} and in full generality 
later on in \cite{Alldridge}, where the author establishes an
equivalence of categories between certain Harish-Chandra modules
and the category of smooth Frech\'et representations of the supergroup
whose module of $K$-finite vectors is Harish-Chandra. 
Our Theorem \ref{th2-intro} can be read as an explicit realization of
such equivalence. These representations appear classically in the
space of holomorphic sections on hermitian spaces. An infinitesimal
realization of such spaces in the supersetting is due first to Serganova
in \cite{ser} and later on to Borthwick et al. in \cite{blr}, though
they do not construct the super Harish-Chandra
modules associated with them.
}

In the present work,
we start by examining the infinitesimal setting, that is the 
Harish-Chandra highest weight supermodules over the complex field for
$\fg$, a complex basic Lie 
superalgebra, $\fg\neq A(n,n)$, $\fg_1\neq 0$.
Let $\fg_r$ be a real form of $\fg$ and let $\fg=\fk\oplus \fp$, be the complex
super Cartan decomposition ($\fk=\fk_0$ the complexification
of the maximal compact subalgebra $\fk_r$ of $\fg_r$). {Assume
$\fh_r \subset \fk_r$, that is $\rk(\fk)=\rk(\fg)$, where
$\fh$, the complexification of $\fh_r$,}
is a Cartan subalgebra (CSA) of both $\fk$ and $\fg$.

The main results for the infinitesimal theory 
are Theorems \ref{theorem4} and \ref{thm4},
that we summarize here.

\medskip
\begin{theorem} \label{th1-intro}
Let $\lambda \in \fh^*$,
$\lambda(H_\al) \in \Z_{\geq 0}$ for all compact roots $\al$.
Let $U^\lambda=\cU(\fg) \otimes_{\cU(\fq)}F$ ($F$ the finite dimensional
representation of $\fk$ associated with $\lambda$, $\fq=\fk\oplus \fp^+$).
Then
\begin{enumerate}
\item $U^\lambda$ is the universal Harish-Chandra supermodule 
of highest weight $\lambda$. 

\item $U^\lambda$ has a unique irreducible quotient, which is the
unique (up to isomorphism) irreducible
highest weight Harish-Chandra supermodule with highest
weight $\lambda$.

\item 
If 
$(\lambda+\rho)(H_\gamma)\leq 0$ for all 
$\gamma\in P_{n}$ and $<0$ for $\gamma$ isotropic, 
then $U^\lambda$ is 
irreducible.
\end{enumerate}
\end{theorem}

The construction of $U^\lambda$ is based on the existence of 
{\sl admissible systems} for $\fg$. These are positive systems
such that the adjoint representation of $\fk$ on $\fp$ stabilizes
$\fp^\pm$, the sum of the positive (negative) non compact
root spaces. In \cite{chf}
is proven the existence of such systems using diagrammatic methods;
here we exhibit them and take a more conceptual
approach (Theorem \ref{existence-admissible-thm}). 
The existence of admissible systems
is instrumental to provide an invariant complex structure
on $\fg_r/\fk_r$, which is the infinitesimal counterpart of the
complex structure on the homogeneous
superspace $G_r/K_r$ ($\fg_r=\Lie(G_r)$, $\fk_r=\Lie(\fk_r)$), 
that we will use to build the global Harish-Chandra
representations.
In the end of this part (Sec. \ref{char-sec}), we obtain a character
formula for the universal Harish-Chandra supermodule.

\medskip
In the second part of the paper 
we  proceed to study the geometric realization of the Harish-Chandra 
representation on the superspace
of holomorphic sections of a vector bundle defined on the 
symmetric superspace $G_r/K_r$. We construct the infinite 
dimensional Harish-Chandra super representations
of a real form $G_r$ of a simple complex Lie supergroup $G$, whose
infinitesimal version are the supermodules constructed in the
first part of the paper. We first consider 
the Fr\'echet
superspace of global sections of a (complex) line bundle \(L^\chi\) 
on the quotient $G/B^+$,
associated with the infinitesimal character $\lambda$ 
(see Theorem \ref{th1-intro}), 
where $B^+$ is a Borel subsupergroup of $G$ associated
with an admissible system $P$.  
In order to prove that \(L^\chi(G_rB^+)\neq \{0\}\),  
we use the existence of a global section of the projection 
\(N^- B^+ \to N^- B^+ /B^+\) which in turn ensures the existence of an 
isomorphism \(\cO(N^-)\to L^\chi (N^- B^+): f\mapsto \cov{f}\).
 Exactly as in the Harish-Chandra's 
theory, under generic conditions, we are able to go, from line 
bundles on \(G/B^+\), to vector bundles over \(G_r/K_r\).
We also obtain the Harish-Chandra decomposition
$P^-KP^+ \subset G$, $\fp^\pm=\Lie(P^\pm)$; however, while in the
ordinary setting $P^\pm$ are abelian subgroups, in the supersetting
this is no longer true. This is due to the fact that
the superalgebras $\fp^\pm$ are not in general abelian. 
Nevertheless, we are able to give a thorough description
of such representations. The main results for this part
are in Theorems \ref{theorem6}, \ref{theorem8} and Corollary \ref{cor-thm6}
that we summarize here:

\medskip
\begin{theorem}\label{th2-intro}
Let $\chi$ be a  character of $B^+ \subset G$ corresponding to
the infinitesimal character $\lambda \in \fh^*$. 
Let
$L^\chi(G_rB^+):=\{f:G_rB^+ \lra \C^{1|1} \, | \, f_T(gb)=\chi_T(b)^{-1}f_T(g)\}$
and assume its reduced part $\widetilde{L^\chi(G_rB^+)}:=\{\tilde{f} \,|\,
f \in  L^\chi(G_rB^+)\}
 \neq 0$. Let $\ell$ denote
the natural left action of $\cU(\fg)$ on the sections in $L^\chi(G_rB^+)$.
Then:
\begin{enumerate}
\item $L^\chi(G_rB^+)$ contains an element $\psi$ which is an
analytic continuation of $\cov{1}$ (the polynomial section corresponding to $1$); 
\item $F^{11}:= \overline{\ell(\cU(\fg))\psi}$ $\subset  L^\chi(G_rB^+)$
is a Fr\'echet $G_r$-supermodule, $K_r$-finite and with $K_r$-finite
part  $\ell(\cU(\fg))\psi) = \cov{\cP}_\lambda$ (the polynomial sections).
\item 
the $K_r$-finite
part  ${\cov{\cP}_{\lambda}}$ is
isomorphic to $\pi_{-\lambda}$ the irreducible representation
with lowest weight $-\lambda$. In particular 
$\lambda(H_\al) 
\in \Z_{\geq 0}$ for all compact positive roots $\alpha$.
\item  $\ell(\cU(\fg))\psi$
is the irreducible 
Harish-Chandra supermodule with highest weight $-\lambda$
with respect to the positive system $-P$.
\end{enumerate}
Vice-versa, if the center of $\fk$ has positive dimension, and 
$\lambda \in \fh^*$ is integral, $K$\textit{-integrable} and
$\lambda(H_\al) \geq 0$ for $\al$ compact positive root, 
then $\widetilde{L^\chi(G_rB^+)} \neq 0$.
\end{theorem}

\medskip
{\bf Acknowledgements.} We wish to thank Professor V. Serganova
for many helpful comments on the infinitesimal theory. We also
thank the anonymous Referee for his invaluable work, which has
greatly helped us to improve this paper.

\section{The infinitesimal theory}

In this section we start with the discussion of the infinitesimal
theory, namely the representation of the pair $(\fg_r,\fk_r)$,
where $\fg_r$ is a real basic Lie superalgebra and $\fk_r$
the maximal compact subalgebra of its even part.

\subsection{Summary of results for basic Lie superalgebras}

Assume 
$\fg$ to be basic classical, $\fg\neq A(n,n)$, 
$\fg_1 \neq 0$, i.e. (see \cite{ka2} Prop. 1.1): 
\begin{align}
\label{list:SLA}
A(m,n)\textrm{ with } m\neq n \,,\, B(m,n)\,,\,
 C(n)\,,\, D(m,n) \,,\, 
D(2,1;\alpha)\,,\,  F(4) \,,\,  G(3)
\end{align}

Let $\fh=\fh_0$ be a Cartan subalgebra of $\fg_0$ and let
$\fg=\fh \oplus \sum_{\al \in \Delta} \fg_\al$ be the
root space decomposition of $\fg$.

$\Delta_0$ and $\Delta_1$ denote the
set of \textit{even} and \textit{odd} roots respectively, where we say that
a root $\al$ is even if $\fg_\al \cap \fg_0 \neq 0$ and similarly for
the odd case.

\medskip

The following proposition  shows that  many properties of 
the Cartan-Killing theory extend to the basic classical Lie superalgebras.

\begin{proposition}[\cite{ka2} Prop. 1.3, \cite{ka} Prop. 2.5.5]
A basic classical Lie superalgebra in list (\ref{list:SLA})
satisfies the following
properties:
\begin{enumerate}
\item
${\dim}\fg_\al=1$, for all $\al\in \Delta$.
\item $[\fg_\al, \fg_\be] = 0$ if and only if $\al$, $\be \in \Delta$, 
$\al+ \be \not\in \Delta$. \\
$[\fg_\al, \fg_\be] = \fg_{\al+\be}$ if $\al$, $\be$, $\al+
\be \in \Delta$.
\item If $\al \in \Delta$, then $-\al \in \Delta$.
\item $(\fg_\al, \fg_\be) =0$ for $\al \neq -\be$. The form $(\,,\,)$
determines a non degenerate pairing of $\fg_\al$ with $\fg_{-\al}$
and $(\,,\,)|_{\fh \times \fh}$ is non degenerate.
\item $[e_\al,e_{-\al}]=(e_\al,e_{-\al})h_\al$, where $h_\al$ is defined by
$(h_\al,h)=\al(h)$ and $e_{\pm \al} \in \fg_{\pm \al}$. 
\item The bilinear form of $\fh^*$ defined by 
$(\lambda, \mu)=(h_\lambda,h_\mu)$ is non degenerate and $W$-invariant,
where $W$ is the Weyl group of $\fg_0$. 
\item $k \al \in \Delta$ for $\al \neq 0$ and $k \neq \pm 1$ if and
only if $\al \in \Delta_1$ and $(\al,\al) \neq 0$. In this case $k = \pm 2$.
\end{enumerate}
\end{proposition}

The last point characterizes a root for which $(\al,\al)=0$: it is
an odd root $\al$ such that $2\al$ is not a root. Such roots 
are called 
\textit{isotropic roots}
Notice that, for example, in $\fsl(m|n)$ all
odd roots are isotropic. 

\medskip

\begin{definition}
Let the notation and the hypotheses be as above.
We define a
\textit{Borel subsuperalgebra}, as a subsuperalgebra $\fb \subset 
\fg$, such that:
\begin{itemize}
\item $\fb_0$ is a Borel subalgebra
of $\fg_0$;
\item $\fb=\fh \oplus \fn^+$, where $\fn^+$ is a nilpotent ideal of $\fb$,
\end{itemize}
and $\fb$ is maximal with respect to these properties (see \cite{musson}
pg 26 for more details).
\end{definition}

Let us fix a Borel subsuperalgebra $\fb$. 
We say that a root $\al$ is \textit{positive} if $\fg_\al \cap \fn^+ \neq (0)$,
we say that it is \textit{negative}
if $\fg_\al \cap \fn^+ = (0)$.
$\fn^+$ is
the span of the positive root spaces
 and we define $\fn^-$ as the span of the
negative ones.
We define the \textit{positive system} $P$ as 
the set of positive roots.
We say that a positive root is \textit{simple} if it is indecomposable,
that is, if we cannot write it as a sum of two positive roots.
Let $\Pi=\{\al_1, \dots , \al_r\}$ be the set of simple roots, we also
call $\Pi$ a \textit{fundamental system}.

\medskip

We have the following result (see \cite{ka2} Prop. 1.5).

\begin{proposition}
Let $\fg$ be a basic classical Lie superalgebra as above. 
Fix a Borel subsuperalgebra and let the notation be as above.
Then
\begin{enumerate}
\item The simple roots $\al_1, \dots, \al_r$ are linearly independent.
\item We may choose elements $e_i \in \fg_{\al_i}$
and $f_i \in  \fg_{-\al_i} $ and $h_i \in \fh$ such that
$\{e_i,f_i,h_i\}_{i\in I}$ 
is the system of generators of $\fg$ satisfying the
relations:
$$
[e_i,f_j]=\de_{ij}h_i, \quad [h_i,h_j]=0, \quad
[h_i,e_j]=a_{ij}e_j, \quad [h_i,f_j]=-a_{ij}f_j.
$$
for a suitable non singular integral matrix $A=(a_{ij})_{i,j \in I}$
(the Cartan matrix of $\fg$).
\item $\fn^\pm$ are generated by the $e_i$'s and $f_i$'s respectively.
\item $\Delta=P \coprod -P$ and $P$ consists of the roots which
are integral positive linear combinations of simple roots. 
\end{enumerate}

\end{proposition}

\medskip
As in the case with the ordinary Lie algebras,
in the theory of finite dimensional representations of $\frak g$, 
a fundamental role is played by the highest weight modules. 
These are defined with respect to the choice of a  CSA $\frak h$ of 
$\frak g$ and a positive system $P$ of roots of $(\frak g, \frak \fh)$. 
They are parametrized by their highest weights.
The universal highest weight modules are known 
as {\it super Verma modules\/}
 and are infinite dimensional.
The irreducible highest weight modules are uniquely determined 
by their highest weights and are the unique irreducible quotients of 
the super Verma modules. The irreducible modules are finite dimensional 
if and only if the highest weight is dominant 
integral and it satisfies a condition which
can be expressed in terms of all of the Borel subsuperalgebras
containing $\fh$ (in the ordinary theory they
are all conjugate, while in the super theory they are not).
Furthermore, 
one obtains all irreducible finite dimensional representations of $\frak g$ 
in this manner (see Refs. \cite{ka, ka2, musson}).

\medskip
Our goal is to describe the infinite dimensional highest weight supermodules,
which are also $\fk$-finite, in the case $\rk(\fk)=\rk(\fg_0)$, so that
the CSA $\fh$ of $\fk$ is also a CSA of $\fg_0^{ss}$, the semisimple part of 
$\fg_0$ (see Sec. \ref{admissible-sec}
for the definition of the compact subalgebra $\fk_r$ and its complexification
$\fk$). It is important
to remark that in the ordinary setting, not every choice of positive
system leads to infinite dimensional highest weight $\fk$-finite 
$\fg$ representations.
The positive systems with this property are called \textit{admissible} and they
are characterized by the presence of \textit{totally positive roots}, 
that is positive non compact roots, which stay positive under the 
adjoint action of $\fk$,
(see Ref. \cite{hc}).

\subsection{Admissible Systems} \label{admissible-sec}

Let $\fg$ be a complex basic classical Lie superalgebra, 
$\fg \neq A(n,n)$ and $\fg_r$ a real form of $\fg$. 
We have the ordinary Cartan decomposition:

\beq
\fg_{r,0}^{ss}=\fk_{r,0}^{ss} \oplus \fp_{r,0}, 
\qquad \fg_0^{ss}=\fk_0^{ss} \oplus \fp_0, \qquad 
\fp_{r,0}=(\fk^{ss}_{r,0})^\perp \label{cdec-0}
\eeq

\noindent where we drop the index $r$ to mean the complexification.

\medskip
The ordinary real Lie algebra $\fg_{r,0}$
may not be semisimple; this happens for the type I Lie
superalgebras, in which case $\fg_{r,0}$ has 
a one-dimensional center. We define $\fk_r$
as $\fk_{r,0}^{ss}$, when $\fg_{r,0}^{ss}=\fg_{r,0}$, that is when $\fg_{r,0}$ is
semisimple, and we define $\fk_r=\fk_{r,0}^{ss}\oplus \fc(\fg_{r,0})$ if
$\fg_{r,0}$ has center $\fc(\fg_{r,0})$. 
We assume $\fk_r$ to be of compact type and reductive.
As usual we drop the index $r$ to mean complexification.

\medskip
Hence, we decompose $\fg$ as:

\beq
\fg =\fk_0 \oplus \fp, \qquad \fp:=\fp_0 \oplus \fg_1.\label{cdec-s}
\eeq

\begin{definition}
We call the pair $(\fg_r, \fk_r)$, described above, a $(\fg_r,\fk_r)$
\textit{pair of Lie superalgebras}.  Note that $\fk_r=\fk_{r,0}$.
We give the same definition for the complex 
pair of Lie superalgebras $(\fg,\fk)$.
\end{definition}

Let us now further assume that  
$$
\rk(\fg_{r,0})=\rk(\fk_r)
$$
Then we can choose a CSA $\fh_r=\fh_{r,0}$ of $\fg_{r,0}$ so that
\begin{align}
\label{eq::standass}
\frak h_r\subset \frak k_r\subset \frak g_r, \qquad
\frak h\subset \frak k\subset \frak g
\end{align}
$\frak h$ is then a CSA of $\frak k$, $\frak g_0$ and 
$\fg$.

\medskip
We fix a positive system $P=P_0 \cup P_1$ of roots for $(\frak g, \frak h)$ and 
write $\alpha >0$ interchangeably with $\alpha \in P$; $P_0$ denotes the
even roots, while $P_1$ the odd roots in $P$. 

\medskip

We say that a root $\alpha$ is \textit{compact} 
(resp. \textit{non-compact} ) 
if $\fg_\alpha\subseteq \fk$ (resp. $\fg_\alpha\subseteq \fp$). 
By \eqref{eq::standass}, we have
\[
\Delta=\Delta_k \cup \Delta_n
\]
where $\Delta_k$ (resp.$\Delta_n$) 
denotes the set of compact (resp. non-compact) roots.

 We call $P_k$ the set of positive compact roots
and $P_{n}$ the set of positive non-compact roots, $P= P_k \coprod P_{n}$.
We define:
$$
\fp^+=\sum_{\al \in P_{n}} \fg_\al \subset \fp, \quad 
\fp^-=\sum_{\al \in -P_{n}} \fg_\al \subset \fp
$$
that is $\fp^+$ (resp. $\fp^-$) is the direct sum of the positive (resp. negative) non-compact root spaces.

\begin{definition}\label{adm-def}
We say that the positive system $P=P_0 \cup P_1$ is \textit{admissible} if: 

\begin{enumerate}
\item $\fp^+$ is $\fk$-stable, that is, $[\fk,\fp^+]\subset \fp^+$; 
\item $\fp^+$ is a Lie subsuperalgebra, that is $[\fp^+,\fp^+] \subset \fp^+$.
\end{enumerate}
\end{definition}
The two conditions are equivalent to saying that 
$\fk \oplus \fp^+$
is a subsuperalgebra of $\fg$ and $\fp^+$ an ideal in $\fk+\fp^+$.
Notice that these conditions imply that
$P_0$ is an admissible system for $\fg_0^{ss}$ the semisimple part of $\fg_0$
and consequently that $\fp^+_0$ is abelian. If $P$ is a positive system
and $P_0$ is admissible, to verify $P$ is admissible, we only need to
check that:
$$
[\fp_1^+,\fp_1^+] \subset \fp_0^+, \qquad [\fk,\fp_1^+]\subset \fp_1^+
$$

\medskip

\medskip

We now want to prove the existence of admissible systems.

\begin{theorem} \label{existence-admissible-thm}
Let $P_0$ be an admissible system for $\fg_0^{ss}$,
the semisimple part of $\fg_0$.
Then $\fg$ has an admissible system $P$ containing $P_0$.
\end{theorem}

\begin{proof}
It is a known fact that if $\fg$ is a simple Lie algebra and $P=P_k\cup P_{n}$ is an admissible system, the only other admissible system containing $P_k$ is $P_k\cup(-P_n)$. Due to the irreducibility of the $\fk$-modules $\fp_0^\pm$ and the fact each weight is multiplicity free, any other admissible system has the form:
\[
P_k^\prime\cup (\pm P_n)
\] 
for some compact positive system $P^\prime_k$. In other words, the noncompact part of an admissible system is  fixed modulo a sign.

Hence, if $\fg$ is a semisimple Lie algebra the admissible systems containing a given $P_k$ are conjugate under the
action of $\Z_2 \times \dots \times \Z_2$, where we have one
copy of $\Z_2$ for each of the simple components of $\fg$. 
Hence, different admissible system are of the form:
\[
\bigcup_{i} (P_{k,i}^\prime \cup  \epsilon_i P_{n,i})
\]
where $P_{k,i}^\prime \cup   P_{n,i}$ is a fixed admissible system 
for the $i$-th factor and $\epsilon_i=\pm1$.

We can apply these considerations to the semisimple part  $\fg_0^{ss}$ 
of the Lie superalgebra $\fg$, and notice that, 
if we show that there exists a super admissible
system containing a particular admissible system $P_0$, then we are done. 
Indeed if we have
\[
P=P_0\cup P_1=\bigcup_{i} (P_0)_{k,i}\cup (P_0)_{n,i}\cup (P_1)_{n,i}
\]
then any other  admissible system for $\fg_0^{ss}$ is of the form
\[
\bigcup_{i} (P^\prime_0)_{k,i}\cup \epsilon_i (P_0)_{n,i}
\]
and hence an easy check shows that
\[
\bigcup_{i} (P_0^\prime )_{k,i}\cup \epsilon_i (P_0)_{n,i} \cup \epsilon_i (P_1)_{n,i}
\]
is an admissible system for $\fg$.

The fact that, for a particular
admissible system $P_0$ for $\fg_0^{ss}$, 
there exists an admissible system of $\fg$ containing  $P_0$, 
will be proven in the remaining part of this section 
by a case by case analysis.

\medskip

The Lie superalgebras of classical type and their real forms
have been classified in \cite{ka} (see also \cite{ser, parker}), 
so we proceed to a case by
case analysis. We briefly recall the ordinary setting.

The only simple basic classical
real Lie algebras $\fg_{r,0}$ giving rise to an hermitian
symmetric space, condition equivalent to have an admissible system, are:

\begin{itemize}
\item AIII. $\fg_{r,0}=\fsu(p,q)$, $\fg_0=\fsl_{p+q}(\C)$,
$\fk_{r,0}=\fsu(p) \oplus \fsu(q) \oplus i\R$, $\fk=\fsl_p(\C) \oplus
\fsl_q(\C) \oplus \C$.

\item BDI ($q=2$). $\fg_{r,0}=\fso_\R(p,2)$, $\fg_0=\fso_\C({p+2})$,
$\fk_{r,0}=\fso_\R(p) \oplus \fso_\R(2)$, $\fk=\fso_\C(p) \oplus \fso_\C(2)$.

\item  DIII. $\fg_{r,0}=\fso^*(2n)$, $\fg_0=\fso_\C({2n})$,
$\fk_{r,0}=\fu(n)$, $\fk=\fgl_n(\C)$.

\item  CI.  
$\fg_{r,0}={\fsp}_n(\R)$, $\fg_0={\fsp}_{n}(\C)$,
$\fk_{r,0}=\fu(n)$, $\fk=\fgl_n(\C)$.

\end{itemize}

We remark that there are two other Lie algebras corresponding to
hermitian symmetric spaces, namely EIII and EVII, 
however neither of them
will appear in the even part of a basic classical Lie superalgebra
$\fg$, $\fg_1\neq 0$. 

\medskip

Let us now proceed and examine the various cases in the super
setting.

\medskip\noindent
{\bf Type $A$}. {\bf Case:} $A(m-1,n-1)=\fsl_{m|n}(\C)$, $m \neq n$. 
The even part is
\[
\fg_0=\fsl_{m}(\C)\oplus \fsl_{n}(\C)\oplus \C
\]

We have 4 possible real forms of $A(m-1,n-1)_0$ that  
correspond to hermitian symmetric spaces. 
They are given by the following table (relative to the
semisimple part of $\fg$):

\medskip
\begin{center}
\begin{tabular}{|c|c|c|}
\hline
type & $\fsl_{m}(\C), m=p+q$  & $\fsl_{n}(\C), n=r+s$ \\
\hline
non-compact  & $\fsu(p,q)$ & $\fsu(r,s)$\\
\hline
compact & $\fsu(p+q)$ & $\fsu(r+s)$ \\
\hline
\end{tabular}
\end{center}

\medskip\noindent
Note that $p$, $q$, $r$, $s$ can take also the value $0$.
\medskip

The only real Lie algebra which is the even part of 
a real form of $\fg$ is: 
$$
\fg_{r,0}=\fsu(p,q)\oplus \fsu(r,s) \oplus \fu(1)
$$
(see \cite{parker}).
We now describe the root system of $\fg=\fsl_{m|n}(\C)$.
Take as CSA $\fh$ the diagonal matrices:
$$
\fh=\{ \, d\,=\, \ddiag(a_1, \dots , a_m,b_1, \dots , b_n)\}
$$
and define $\ep_i(d)=a_i$, $\de_j(d)=b_j$, for $i=1, \dots, m$
and $j=1, \dots, n$. Choose the simple system:
$$
\Pi=\{ \ep_1 - \ep_2, \dots , \ep_{m-1}-\ep_m, \ep_{m}-\de_1,
\de_1 - \de_2, \dots , \de_{n-1}-\de_n\}
$$
We have 3 non-compact simple roots, 2 even and 1 odd:
$$
\ep_{p} - \ep_{p+1}, \quad \de_{r}-\de_{r+1}, \quad   \ep_{m}-\de_1
$$
According to the simple system, we have: 
\begin{align*}
P_k=& \{\ep_i-\ep_j\mid 1\leq i<j\leq p\}\cup \{\ep_i-\ep_j\mid p+1\leq i<j\leq m\}\\
&\cup\{\de_i-\de_j\mid 1\leq i<j\leq r\}\cup \{\de_i-\de_j\mid r+1\leq i<j\leq n\}\\ 
P_{n,0}=& \{\ep_i-\ep_j\mid 1\leq i\leq p<j\leq m\}\cup \{\de_i-\de_j\mid 1\leq i\leq r<j\leq n\}\\ 
P_{n,1}=& \{\ep_i-\de_j\mid 1\leq i\leq m\,,\, 1\leq j \leq n\}
\end{align*}

The following checks are immediate: 
\begin{enumerate}
\item $[\fk_0,\fp_0^+] \subset \fp_0^+$, 
$[\fk_0,\fp_1^+] \subset \fp_1^+$. 
\item $[\fp_0^+, \fp_0^+]=0$,  $[\fp_0^+, \fp_1^+] \subset \fp_1^+$, 
 $[\fp_1^+, \fp_1^+]=0$.
\end{enumerate}

\noindent
Hence we have produced an admissible system.

\bigskip
\noindent
{\bf Type $B$}. 
{\bf Case:} $B(m,n)=\rosp_\C(2m+1|2n)$, $m \neq 0$.
The even part is 
\[
B(m,n)_0=\rosp_\C(2m+1|2n)_0=\fso_\C(2m+1) \oplus {\fsp}_n(\C)
\]

Reasoning as in the previous cases, that is reasoning on the
ordinary setting and knowing the classification of the real forms there, 
we have only four possibilities for the choice of
the real form of $B(m|n)_0$:
 
\medskip
 
 \begin{center}
\begin{tabular}{|c|c|c|}
\hline
type & $\fso_\C(2m+1)$  & $ {\fsp}_n(\C)$ \\
\hline
non-compact  & $\fso_\R(2,2m-1)$ & ${\fsp}_n(\R)$\\
\hline
compact & $\fso_{\R}(2m+1)$ & ${\fsp}(n)$ \\
\hline
\end{tabular}
 \end{center}
 
\medskip

The only case for which the real form $\fg_{r,0}$ extends to
a real form of the Lie superalgebra $\fg$ is 
\[
\fg_{r,0}=\fso_\R(2,2m-1) \oplus {\fsp}_n(\R)
\] 
(see \cite{parker}). The corresponding maximal compact subalgebra is 
\[
\fk_0=\fso_\R(2) \oplus \fso_\R(2m-1)\oplus \fu(n)
\]

Using the same notations as in \cite{Knapp},  we choose the simple system:
$$
\Pi=\{ \ep_1 - \ep_2\,, \dots ,\, \ep_{m-1}-\ep_m\,,\, \ep_m\,,\, 
\de_1 - \de_2\,, \dots ,\, \de_{n-1}-\de_n\,,\, \de_{n}-\ep_1\}
$$
We have one simple non-compact even root $\ep_1 - \ep_2$ 
and one non-compact simple odd root: $\de_{n}-\ep_1$.
So we have: 
\begin{align*}
P_{n,0} = & \{\ep_1 \pm \ep_j\mid 1<j\leq m  \}\cup\{\ep_1\} 
\cup \{\de_i+\de_j\mid 1\leq i,j \leq n \} \\ 
P_{n,1} = & \{\de_i \pm \ep_j\mid 1\leq i\leq n \,,\, 
1\leq j\leq m\}\cup \{\delta_i\mid 1\leq i \leq n\}\\ 
P_{k}= &
\{\ep_i \pm \ep_j\mid 1<i<j\leq m \}\cup 
\{\ep_i\mid 1< i \leq m \} \cup\{ \de_i - \de_j\mid1\leq i<j\leq n\}
\end{align*}

In order to check the $\ad(\fk_0)$ invariance of $\fp_0^+$
and $\fp_1^+$, we need to
verify that summing one of the roots in {$P_k \cup -P_k$} 
to one in {$P_{n,0}$ or $P_{n,1}$} 
we remain in  {$P_{n,0}$ or $P_{n,1}$}. 
The check is straightforward. Similarly one verifies that
$\fp^+=\fp_0^++ \fp_1^+$ is a Lie subsuperalgebra, hence we
have produced an admissible system.

\bigskip
\noindent
Type $B$. {\bf Case:} $B(0,n)=\rosp_\C(1|2n)$.
In this case, the even part is 
\[
B(0,n)_0=\rosp_\C(1|2n)_0
\cong {\fsp}_n(\C)
\]
The only real form of interest to us is:
\[
\fg_{r,0}={\fsp}_n(\R)
\] 
corresponding to the maximal compact subalgebra $\fk_0= \fu(n)$.

Choose the simple system:
$$
\Pi=\{ \de_1 - \de_2, \dots , \de_{n-1}-\de_n, \de_n\}
$$
We have one simple non-compact odd root $\de_n$. 
So we have: 
\begin{align*}
P_{n,0} = &\{ \de_i+\de_j\mid 1\leq i , j\leq n \}\\ 
P_{n,1}= &\{ \delta_i\mid 1\leq i \leq n \}\\ 
P_{k} =&\{
 \de_i - \de_j \quad 1\leq i<j\leq n \}
\end{align*}

Both the checks for properties (1) and (2) in Def. \ref{adm-def}
are immediate.

\bigskip

\noindent
{\bf Type $D$}. 
{\bf Case:} $D(m,n)={\rosp_\C(2m|2n)}$. 
The even part is given by: 
\[
D(m,n)_0={\rosp_\C(2m|2n)_0}=\fso_\C(2m) \oplus {\fsp}_n(\C)
\]
The  possibilities for the choice of
the real form of $D(m,n)_0$ are:
 
 \medskip
 
  \begin{center}
\begin{tabular}{|c|c|c|}
\hline
type & $\fso_\C(2m+1)$  & ${\fsp}_n(\C)$\\
\hline
non-compact  & $\fso_\R(2,2m-2)$ & ${\fsp}_n(\R)$\\
\hline
compact & $\fso_{\R}(2m)$ & ${\fsp}(n)$ \\
\hline
\hline
non-compact  & $\fso^\ast(2m)$ & ${\fsp_{n}}(\R)$\\
\hline
compact & $\fso_{\R}(2m)$ & ${\fsp}_n(\R)$ \\
\hline
\end{tabular}
 \end{center}
 
\medskip

Again by \cite{parker} we have that the only possibilities 
for $\fg_{r,0}$ are:
\[
\fg_{r,0}=\fso_\R(2,2m-2)\oplus {\fsp}_n(\R)
\]
or
\[
\fg_{r,0}=\fso^\ast(2m)\oplus {\fsp}_n(\R)
\]

\medskip

The corresponding maximal compact subalgebra in the first case is
\[
\fk_{r,0}= \fso_\R(2) \oplus \fso_\R(2m-2) \oplus \fu(n)
\]

\medskip

Choose the simple system:
$$
\Pi=\{ \ep_1 - \ep_2, \dots , \ep_{m-1}-\ep_m, \ep_{m-1}+\ep_m, 
\de_1 - \de_2, \dots , \de_{n-1}-\de_n, \de_{n}-\ep_1\}
$$
We have one simple non-compact even root $\ep_1 - \ep_2$ 
and one non-compact simple odd root: $\de_{n}-\ep_1$. 
\begin{align*}
P_{n,0}= & \{\ep_1 \pm \ep_j\mid 1<j\leq m \}
\cup\{ \de_i+\de_j\mid1\leq i,j\leq n \}\\ 
P_{n,1}= & \{ \de_i \pm \ep_j\mid 1\leq j\leq m \,,\, 1\leq i \leq n \}\\ 
P_{k}= &
\{ \ep_i \pm \ep_j\mid 1<i<j\leq m \}\cup
\{ \de_i - \de_j \mid 1\leq i<j\leq n\}
\end{align*}

The calculation of $\ad(\fk_0)$-stability is exactly as before
and similarly for the verification of the property (2) 
in Def. \ref{adm-def}.

\medskip
We now go to the second case. 

The maximal compact subalgebra is
\[
\fk_{r,0}= \fu(m) \oplus \fu(n)
\]

\medskip

Choose the same simple system as before:
$$
\Pi=\{ \ep_1 - \ep_2, \dots , \ep_{m-1}-\ep_m, \ep_{m-1}+\ep_m, 
\de_1 - \de_2, \dots , \de_{n-1}-\de_n, \de_{n}-\ep_1\}
$$
Now the simple non-compact even root $\ep_{m-1} + \ep_m$
while 
the non-compact simple odd root is as before $\de_{n}-\ep_1$. 
\begin{align*}
P_{n,0} = & \{\ep_i + \ep_j\mid 1\leq i <j \leq m\}
\cup\{ \de_i+\de_j\mid 1\leq i,j\leq n \}\\ 
P_{n,1} = & \{\de_i \pm \ep_j\mid 1\leq i \leq n\,,\, 1\leq j \leq m \}\\ 
P_{k} = &\{ \ep_i - \ep_j\mid 1<i<j\leq m \}\cup 
\{\de_i - \de_j\mid  1\leq i<j\leq n\}
\end{align*}

The calculation of $\ad(\fk_0)$-stability is exactly as before
and similarly for the verification of the property (2) 
in Def. \ref{adm-def}.

\bigskip
\noindent
{\bf Type $C$}. 
{\bf Case:} $C(n)={\rosp_\C(2|2n-2)}$. 
We have only one possible real form extending to a real form of
the whole $\fg$ namely:
$$
\fg_{r,0}=\fso_\R(2) \oplus {\fsp}_{n-1}(\R)
$$
Choose the simple system:
$$
\Pi=\{ \ep - \de_1, 
\de_1 - \de_2, \dots , \de_{n-2}-\de_{n-1}, 2\de_{n-1}\}
$$

We have one simple non-compact odd root $\ep_1 - \de_1$ 
and one non-compact simple even root: $2\de_{n-1}$. 
\begin{align*}
P_{n,0}= & \{ \de_i+\de_j \mid 1\leq i,j\leq n-1 \}\\ 
P_{n,1}= & \{\epsilon\pm \delta_j\mid 1\leq j \leq n-1\}\\ 
P_{k}=& \{\de_i - \de_j\mid 1\leq i <j\leq n-1\} 
\end{align*}

The verification of the two properties $(1)$ and $(2)$ listed
above is the same as in the $B$ case.

\bigskip

We now examine the exceptional Lie superalgebras of
classical type.

\medskip\noindent
{\bf Case:} 
$D(2,1;\alpha)$. 
We have two possible real forms of 
$D(2,1;\alpha)_0= \fsl_2(\C) \oplus \fsl_2(\C) \oplus\fsl_2(\C)$
admitting an extension to the whole $D(2,1;\alpha)$ (see \cite{parker})
namely:
$$
\fg_{r,0}=\fsl_2(\R) \oplus \fsl_2(\R) \oplus\fsl_2(\R), \qquad
\hbox{and} \qquad
\fg_{r,0}=\fsl_2(\R) \oplus \fsu(2) \oplus \fsu(2).
$$
We first examine the case
$\fg_{r,0}=\fsl_2(\R) \oplus \fsl_2(\R) \oplus\fsl_2(\R)$.
The root system is:
$$
\Delta_0=\{\pm 2\ep_i \mid i=1,2,3\}, \qquad 
\Delta_1=\{\pm \ep_1 \pm \ep_2\pm \ep_3\}
$$
All roots are non-compact and the positive roots corresponding to
the simple system:
$$
\Pi=\{\ep_1 + \ep_2 + \ep_3, -2\ep_2, -2\ep_3 \}
$$
form an admissible system. In fact the only thing to check
is that $[\fp_1^+,\fp_1^+] \subset \fp^+_0$,
where ${P_{n,1}=\{\ep_1 \pm \ep_2 \pm \ep_3\}}$ and ${P_{n,0}}$ consists
of the positive even roots.
This is immediate. 

\medskip

Consider now the case $\fg_{r,0}=\fsl_2(\R) \oplus\fsu(2) \oplus \fsu(2)$
and fix the positive system:
$$
\Pi=\{\ep_1 + \ep_2 + \ep_3, -2\ep_2, -2\ep_3 \}
$$
where $\ep_1+ \ep_2 + \ep_3$ is non-compact, while 
$-2\ep_2$, $-2\ep_3$ are compact.
We have that $\fp^+$ is spanned by 
the root spaces corresponding to the roots:
$$
2\ep_1, \quad \ep_1+\ep_2+\ep_3, \quad  \ep_1-\ep_2-\ep_3, \quad  
\ep_1-\ep_2+\ep_3, \quad \ep_1+\ep_2-\ep_3.  
$$
With such a choice we have that $\fp^+$ is $\ad(\fk_0)$-stable
and $[\fp_1^+,\fp_1^+]\subset \fp_0^+$. 

\bigskip

\medskip\noindent
{\bf Case:} $F(4)$.  
The root system is:
$$
\Delta_0=\{\pm \ep_i \pm \ep_j, \, \pm 
\ep_i, \pm \delta, \quad i=1,2,3\}, \qquad 
\Delta_1=\{1/2(\pm \ep_1 \pm \ep_2\pm \ep_3 \pm \de)\}
$$
$\ep_1$, $\ep_2$, $\ep_3$ are the roots corresponding to $\fso_7(\C)$,
while $\de$ corresponds to $\fsl_2(\C)$.
Choose
the simple system (refer to \cite{ka} pg 53):
$$
\Pi=\{1/2(\ep_1 + \ep_2 + \ep_3+\de), -\ep_1, \ep_1-\ep_2, \ep_2-\ep_3 \}
$$
We have two real forms of the even part 
$F(4)_0=\rsl_2(\C) \oplus \fso_7(\C)$ which extend to the whole $F(4)$:
$$
\fg_{r,0}=\fsl_2(\R) \oplus \fso_\R(7), \qquad
\fg_{r,0}=\fsu(2) \oplus \fso_\R(2,5)
$$
Let us first examine 
$$
\fg_{r,0}=\fsl_2(\R) \oplus \fso_\R(7), \qquad \fk_{r}=\fso_\R(2)
\oplus \fso_\R(7)
$$
We have that $1/2(\ep_1 + \ep_2 + \ep_3+\de)$ is the only 
non-compact simple root.
With such a choice: 
\begin{align*}
P_{k}&=\{\pm \ep_i - \ep_j, \, 1 \leq i<j \leq 3\} \, \cup \, 
\{-\ep_i, \, 1\leq i \leq 3 \} \\ 
P_{n,0}&=\{\de\}, \\ 
P_{n,1}&=\{1/2(\pm \ep_1 \pm \ep_2 \pm \ep_3+\de)\} 
\end{align*}
One can easily check that such $\fp^+$ is $\ad(\fk)$-invariant
and $[\fp^+_1, \fp_1^+]\subset \fp_0^+$.

\medskip

We now consider the case:
$$
\fg_{r,0}=\fsu(2) \oplus \fso_\R(2,5), \qquad 
\fk_{r}=\fsu(2) \oplus \fso_\R(2) \oplus \fso(5)
$$
The simple root system described
above is not suitable, since the corresponding even positive system 
is not admissible for the given real form 
(notice that both the roots $\ep_1-\ep_2$ and $-\ep_3$ are
positive and compare with case $B(m,n)$ discussed above). 
In order to obtain the
correct simple system we need to transform
it using the (unique) element $w$ of the Weyl group such
that $w(\Pi_0)=\Pi_0'$, where $\Pi_0'$ leads to an admissible
system for the real form $\fsl_2(\R) \oplus \fso_\R(2,5)$ of
$F(4)$:
$$
\Pi_0=\{-\ep_1, \ep_1-\ep_2, \ep_2-\ep_3 \} \mapsto
\Pi_0'=\{\ep_3, \ep_1-\ep_2, \ep_2-\ep_3 \}
$$
Now, taking the image of the simple system $\Pi$ under $w$ we obtain:
$$
\Pi':=w(\Pi)=\{\ep_3, \ep_1-\ep_2, \ep_2-\ep_3, (1/2)(w(\ep_1+\ep_2+\ep_3)+\de) 
\}
$$
The odd positive roots are then characterized by having $+\de$ and
not $-\de$ in their expression.
We take as the simple non-compact roots:
$$
1/2(w(\ep_1 + \ep_2 + \ep_3)+\de), \qquad \ep_1-\ep_2.
$$ 
With such a choice we obtain:  
\begin{align*}
P_{n,0}&=\{\de, \, \ep_1, \, \ep_1 \pm \ep_j, \, 1 \leq i<j \leq 3\} \\ 
P_{n,1}&=\{1/2(\pm \ep_1 \pm \ep_2 \pm \ep_3+\de)\}  \\ 
P_k&=\{\ep_i \pm \ep_j, \, 1 < i<j \leq 3\} \, \cup \,
\{\ep_i, \, i=2,3\} 
\end{align*}
$\fp^+_0$ is  $\ad(\fk)$-invariant (this is an immediate check, but it also
comes from the ordinary theory), while $\fp^+_1$ is  $\ad(\fk)$-invariant
since no roots in $P_k$ contain $\de$. Finally 
$[\fp^+_1, \fp_1^+]\subset \fp_0^+$.

\bigskip\noindent
{\bf Case:} $G(3)$. The only real form of $G(3)_0$ with
an admissible system is:
$$
\fsl_2(\R) \oplus {\mathcal{G}}_2
$$ 
where ${\mathcal G}_2$ is the compact form of $G_2$. 
Choose
the simple system (refer to \cite{ka} pg 53):
$$
\Pi=\{\de+\ep_1, \ep_2, \ep_3-\ep_2\}
$$
where the linear functions $\ep_i$ correspond to $G_2$, $\ep_1+\ep_2+\ep_3=0$
and $\de$ to $A_1$. The only simple non-compact root is $\de+\ep_1$.
We have: 
\begin{align*}
P_{k}&=\{-\ep_1, \ep_i, \,  1< i\leq 3\} \cup 
\{\ep_3-\ep_2, \, \ep_2-\ep_1, \, \ep_3-\ep_1\} \\ 
P_{n,0}&=\{2\de\} \\ 
P_{n,1}&=\{\de, \, \de\pm\ep_i, \, 1\leq i\leq 3\}
\end{align*}
The properties (1) and (2) of Def. \ref{adm-def} are verified by
an easy calculation.

\medskip
No other Lie superalgebra of classical type satisfying our hypothesis
admits a real form whose even part corresponds to
an hermitian symmetric space,
hence our case by case analysis ends here.
\end{proof}

\subsection{Harish-Chandra modules} \label{hcmod-sec}

Consider {a} $(\fg,\fk)$ pair of Lie superalgebras as above, with
${\rm rk}\ \fg={\rm rk}\ \fk$ and
$\fk=\fk_0$. We can choose a CSA $\frak h=\fh_0$ so that
$$
\frak h\subset \frak k\subset \frak g
$$
and $\frak h$ will be a CSA of both $\frak k$ and $\frak g$. 
We fix a positive system $P=P_0 \cup P_1$ of roots for $(\frak g, \frak h)$.

\begin{definition}
Let the complex super vector space $V$ be a $\fg$-module. 
We say that $V$ is a \textit{$(\fg, \fk)$-module} if 
$$
V=\sum_{\theta \in \Theta} V(\theta)
$$ 
where the sum is algebraic and direct, $\Theta$ denotes the
set of equivalence classes of the finite dimensional irreducible
representations of $\fk$ and $V(\theta)$ is the sum of all
representations occurring in $V$ and 
belonging to the class 
$\theta \in \Theta$.
We say that the  $(\fg, \fk)$-module $V$ is an \textit{Harish-Chandra
module} (or \textit{HC-module} for short) if each 
$V(\theta)$ is finite dimensional and $V$ is finitely generated as
$\cU(\fg)$ module. 
\end{definition}

We are interested in highest weight modules (with respect to $P$) 
which are also HC-modules (see Ref. \cite{musson} for an
exhaustive introduction to highest weight modules).

\begin{proposition} \label{infinitesimalhc} 
Let $U$ be a highest weight $\fg$-module with highest weight $\lambda$
and highest weight vector $v$. The following are equivalent:

\begin{enumerate}
\item $\dim(\cU(\fk)v)<\infty$;
\item $U$ is a $(\fg,\fk)$-module;
\item $U$ is a{n} HC-module.
\end{enumerate}

\noindent If these conditions are true, then $\cU(\fk)v$ is an irreducible
$\fk$-module.
\end{proposition}

\begin{proof}  
The proof is very similar to the classical case; we shall nevertheless
rewrite it for clarity of exposition. By the ordinary theory we know
that the action of $\fk$ (which is an ordinary reductive Lie algebra) 
is decomposable if and only if  the center $\fc$ of $\fk$ acts semisimply. 
Given $\fc \subset \fh$, if $U$ is an highest weight module
(on which $\fh$ acts diagonally) we have that $\fc$ acts semisimply on $U$.
We now show $(1) \implies (2)$. Let $U^\fk$ denote the $\fk$-finite
vectors in $U$ (i.e., those vectors lying in a finite dimensional
$\fk$-stable subspace). It is easy to check $U^\fk$ is a submodule and
since the highest weight vector $v \in U^\fk$, we have $U=U^\fk$ and
this proves $(2)$.  We now show  $(2) \implies (3)$. 
According to the previous definition, we need to show
that $U(\theta)$ is finite dimensional. By contradiction, assume this
is not the case and let $\mu$ a weight of $U(\theta)$. Such
a weight occurs with infinite multiplicity, and we have 
$\dim(U(\theta)_\mu)=\infty$. This is in contradiction with the 
well known fact that in a highest weight module the weights occur 
with finite multiplicities. 
$(3) \implies (1)$ is straightforward.
We now go about the proof of the irreducibility of
$\cU(\fk)v$. As we remarked at
the beginning $\fc$ acts as scalar: $cw=\lambda(c)w$, for $c \in \fc$
and $w \in  \cU(\fk)v$. So we have $\cU(\fk)v=\cU(\fk')v$ with
$\fk'=[\fk,\fk]$. $\cU(\fk')v$ is a finite dimensional highest weight
module for the semisimple Lie algebra $\fk'$ and consequently
it is irreducible. 
\end{proof}

We now want to define the universal super HC-module of highest weight 
$\lambda$. Choose $P=P_k {\cup} P_{n}$, a positive admissible system.

\medskip 
Let $\lambda\in \fh^\ast$ be such that $\lambda(H_\alpha)$ is an integer 
$\ge 0$ for all $\alpha\in P_k$. Let $F=F_\lambda$ be the irreducible 
finite dimensional module for $\fk$ of highest weight $\lambda$. 
Note that $\lambda(H_\beta)$ can be arbitrary for positive non-compact 
roots $\beta$. Write 
$$
\fq=\fk \oplus \fp^+.
$$
Recall that $[\fk, \fp^+]\subset \frak p^+$ and so we can turn 
$F$ into a left $\fq$-module by letting $\fp^+$ act trivially. 
Define

\begin{equation}
U^\lambda=\cU(\fg)\otimes _{\cU(\fq)}F\label{def::univHC}
\end{equation}

\noindent and view $U^\lambda$ as a $\cU(\fg)$-module by left action
$
a(b\otimes f)=ab\otimes f.
$

Let

\begin{equation}
\rho=\frac{1}{2}\sum _{\alpha \in P_0}\alpha-\frac{1}{2}\sum _{\alpha \in P_1}\alpha. 
\label{eq::defdelta}
\end{equation}

\begin{theorem} \label{theorem4}
Let the notation be as above.
\begin{enumerate}
\item $U^\lambda$ is the universal HC-module of highest weight $\lambda$. 

\item $U^\lambda$ has a unique irreducible quotient, which is the
unique (up to isomorphism) irreducible
highest weight Harish-Chandra module with highest
weight $\lambda$.

\end{enumerate}

\end{theorem} 

\begin{proof}
$(1)$. Let $f^\lambda$ be the highest weight vector for $F$. Then
$1 \otimes f^\lambda$ is such that:
$$
X_\al (1 \otimes f^\lambda)=0, \quad \al >0, \qquad 
\cU(\fg)(1 \otimes f^\lambda) = U^\lambda
$$
hence $1 \otimes f^\lambda$ is the highest weight vector of
$U^\lambda$. To prove universality, let $V$ be a highest weight
HC-module with highest weight vector $v$.The map
$$
F \lra V, \quad uf^\lambda \, \mapsto \, uv, \qquad (a \in \cU(\fg), \, v \in F)
$$
extends to a linear map $L:\cU(\fg) \otimes_\C F \lra V$ which is
onto.
Since $L(au \otimes f)-L(a\otimes uf)=0$ for all
$a \in \cU(\fg)$,  $u \in \cU(\fq)$, $f \in F$, it follows that $L$ descends
to a map $U^\lambda \lra V$ which is obviously a $\cU(\fg)$-module map.

\medskip
$(2)$ It follows from the standard theory of highest weight modules
that $U^\lambda$ has a unique irreducible quotient, which is a highest
 weight HC-module of highest weight $\lambda$. It is the unique
irreducible highest weight HC-module of highest weight $\lambda$ by
universality (point 1).

\end{proof}

We shall now study the structure of $U^\lambda$ as a $\fq$-module for 
arbitrary $\lambda$ with $\lambda(H_\alpha)$ an integer $\ge 0$ for all 
$\alpha\in P_k$. For this we need a standard lemma.

\begin{lemma} \label{lemma4a}
Let $g$ be a field and $A, B$ algebras over $g$. Suppose $B\subset A, A$ is a free right $B$-module, $F$ a left $B$-module, and $V=A\otimes _BF$. If $(a_i)$ is a free basis for $A$ as a right $B$-module, and $L=\sum _ig.a_i$, then the map taking $\ell\otimes_gf$ to $\ell\otimes_Bf$ is a linear isomorphism of $L\otimes_gF$ with $V$.
\end{lemma}

We regard $\cU(\fp^-)\otimes F$ as a $\cU(\fp^-)$-module by 
$a,b\otimes f\mapsto ab\otimes f$. Since $\fp^-$ is stable under 
${\rm ad}\,\fk$ we may view 
$\cU(\fp^-)\otimes F$ as a $\fk$-module also. 

\begin{corollary} 
\label{corol::intertwcorol}
The map $\varphi : a\otimes f\mapsto a\otimes_{\cU(\fq)}f$ is a linear isomorphism of $\cU(\fp^-)\otimes F$ with $U^\lambda$ that intertwines both the actions of $\cU(\fp^-)$ and $\cU(\fk)$. In particular, $U^\lambda$ is a free $\cU(\fp^-)$-module with basis $1\otimes _{\cU(\fq)}f_j$ where $(f_j)$ is a basis for $F$.
\end{corollary}
\begin{proof} Since $\fg=\fp^-\oplus \fq$ it follows that $a\otimes b\mapsto ab$ is a linear isomorphism of $\cU(\fp^-)\otimes \cU(\fq)$ with $\cU(\fg)$. It is clear from this that $\cU(\fg)$ is a free right $\cU(\fq)$-module, and that any basis of $\cU(\fp^-)$ is a free right $\cU(\fq)$-basis for $\cU(\fg)$. Lemma 
\ref{lemma4a} now applies and shows that $\varphi$ is an isomorphism. It obviously commutes with the action of $\cU(\fp^-)$. The verification of the commutativity with respect to $\fk$ is also straightforward.\end{proof}

\subsection{Irreducibility of the universal HC-module}

In this section we want to prove the following theorem which 
gives a sufficient condition for the irreducibility of the universal HC-module. 
For similar modules in the classical setting, 
see Theorem 3 in \cite{hc}, IV (1955), pg 770.

\medskip\noindent
Let our hypotheses and notation be as in Sec. \ref{hcmod-sec}.

\begin{theorem}\label{thm4}
Let $\fg$ be one of the complex basic Lie superalgebras in the list \ref{list:SLA}, and let $U^\lambda$ be the universal Harish-Chandra module, with highest weight $\lambda$ associated
with the (finite dimensional) representation $F$ of $\fk$, as defined by~\eqref{def::univHC}. Let $\rho$ be  as in eq.~\eqref{eq::defdelta}.
If 

\beq\label{cond-irr}
(\lambda+\rho)(H_\gamma)\leq 0\quad \hbox{for all} \quad 
\gamma\in P_{n} \quad  \hbox{and} \quad <0 \quad \hbox{for} \quad \gamma
\quad \hbox{isotropic}, 
\eeq

\noindent then $U^\lambda$ is irreducible. 
\end{theorem}

The proof of this theorem relies on a result stated by V. Kac in  \cite{ka4}
 and proved by M. Gorelik in \cite{go}, that we give 
{here, as Theorem \ref{theo::KacGor},} without proof 
(see also \cite{musson} Sec. 13.2). 
Here $W$ denotes the Weyl group of $\fg_0$ and 
$S(\fh)^W$ denotes the set of $W$-invariant symmetric tensors on $\fh$.
In the following we will use the identification of  $S(\fh)$ with 
the superalgebra of polynomials $\textrm{Pol}(\fh^\ast)$ without mention.

\begin{theorem}
\label{theo::KacGor}
The Harish-Chandra isomorphism
\begin{align*}
\beta\colon Z(\fg)\to S(\fh)^W
\end{align*}
for basic Lie superalgebras
identifies the center
$Z(\fg)$ of the universal enveloping
algebra with the subalgebra $I(\fh)$ of $S(\fh)^W$:
$$
I(\fh)=\{ \phi \in S(\fh)^W \, | \, \phi(\lambda + t \al)=
\phi(\lambda), \,\forall\, \lambda \in \langle \al \rangle^\perp, 
\, \al \, \hbox{isotropic},\, \forall\,t\in \C\}
$$
(a root $\al$ is {\sl isotropic} if $\langle \al, \al \rangle=0$). 
\end{theorem}

This theorem 
allows us to show
that when a weight $\lambda$ is {\sl typical} that is
$$
\langle \lambda+\rho, \al \rangle \neq 0\quad \forall\, \alpha\, 
\mathrm{ isotropic}
$$
all of the {maximal weights} (with respect to
a highest weight vector in a submodule) in the highest weight representation
of highest weight $\lambda$ are conjugate under the affine  action of the
Weyl group, that we denote as $s . \lambda$.
Notice that the condition (\ref{cond-irr}) 
implies that the weight $\lambda$ in Theorem \ref{thm4}
is typical.

We now recall few results (see \cite{musson} Ch. 13).

\begin{lemma}
Let $g = \prod_{\al \, isotropic} h_\al \in S(\fh)$,
where $h_\al \in \fh$ is defined by the property
$h_\al(\mu)=\langle \mu, \al \rangle$. Then:
\begin{enumerate}
\item $\C+g S(\fh)^W \subset I(\fh) \subset S(\fh)^W$.

\item $I(\fh)_g=S(\fh)^W_g \supset S(\fh)^W$.
\end{enumerate}
where $I(\fh)_g$ and $S(\fh)^W_g$ denote respectively  
the localizations of $I(\fh)$ and $S(\fh)$ at the set
\(
G\coloneqq \{ g^k \mid k\geq 0 \}.
\)
\end{lemma}

\begin{proof} (1). Clearly $\C\subset I(\fh)$ and any easy check  shows that  $g S(\fh)^W$ is a subalgebra contained in $I(\fh)$.  Since $I(\fh)$ is a subalgebra we have (1).
As for (2), it follows from the inclusions in (1) by
noticing that  the
localization at $G$ is the same for $\C+g S(\fh)^W$
and $S(\fh)^W$.
\end{proof}

We now introduce the {\sl infinitesimal character}
\begin{align*}
\chi_\lambda: Z(\fg) & \lra \C
\end{align*}
defined as $\chi_\lambda(z)=(\beta(z))(\lambda+\rho)$, 
where $\beta$ is the Harish-Chandra isomorphism.

\begin{prop}\label{lem1}
Let $\lambda$ be typical.
Then $\chi_\lambda=\chi_\mu$ implies
$\mu=s.\lambda$ for some $s\in W$.
 \end{prop}

\begin{proof} The infinitesimal character $\chi_\lambda$ may be
thought (via the HC isomorphism) as defined on $I(\fh)$ since
$Z(\fg) \cong I(\fh)$, by Theorem~\ref{theo::KacGor}. 
Since $\lambda$ is typical, we have
${\chi_{\lambda}(g) \neq 0}$, hence we may extend
(uniquely) ${\chi_{\lambda}}$ to $I(\fh)_g =S(\fh)_g^W \supset S(\fh)^W$. Hence
${\chi_{\lambda}=\chi_{\mu}}$ on $S(\fh)^W$. This implies {$\mu=s.\lambda$}
(from classical considerations, see, for example, 
\cite{Knapp}, Ch. 5).
\end{proof}

We now approach the proof of \ref{thm4} with some lemmas.

\begin{lemma} \label{lem2}
Let $\lambda$ be the highest weight as in \ref{thm4} and let $\mu$
be a maximal weight of $U^\lambda$ with respect to the admissible system
$P=P_k \coprod P_n$. Then 

\begin{equation}
P^-=P_k \coprod -P_n\label{eq:newpossyst}
\end{equation}
 is also a positive system and
$\lambda > \mu$ with respect to $P^-_0$.
\end{lemma}

\begin{proof} The fact $P^-$ is a positive system comes from
the fact that $P^-=-s_0P$, for $s_0$ the longest element in $W_k$.
In fact $-s_0(P)=-s_0(P_k) \coprod -s_0(P_n)$ {and} $-s_0(P_k)=P_k$,
while $-s_0(P_n)=-P_n$. This is because $P$ is chosen admissible,
hence the roots in $P_n$ represent the weights of the adjoint representation
of $\fk$ on $\fp^+$, 
hence they are permuted by the action of $W_k$ the
Weyl group of $\fk$.

\medskip
Now we turn to the second statement. Let $v_\lambda$ and $v_\mu$ 
maximal vectors with weight $\lambda$ and $\mu$ respectively.
Hence $z \in Z(\fg)$ acts as
multiplication by the scalars $\chi_\lambda(z)$ and  $\chi_\mu(z)$ respectively. 
Since $v_\lambda$ is the highest weight {vector}, 
$z$ acts as $\chi_\lambda(z)$ on 
the whole $U^\lambda$. Hence  the two scalars $\chi_\lambda(z)$ and 
$\chi_\mu(z)$ have to be the same (since
$\mathcal{U}(\fg) v_\mu $ is a submodule of $U^\lambda$).

Then, by the previous lemma $s.\lambda=\mu$, that is
$\mu+\rho=s(\lambda+\rho)$,
for $s \in W$. 
Since the hypothesis of 
\ref{thm4} guarantee that $\lambda+\rho$ is dominant with respect
to $P^-$, by
usual facts on groups of reflections (see \cite{vsv1} Appendix to Ch. 4)
we have that 
$\lambda-\mu=\lambda+\rho-s(\lambda+\rho)$ is sum of
simple roots, but since $\lambda - \mu$ is even, it will be
the sum of simple roots of $P_0^-$.
\end{proof}

We now make some remarks on the simple systems of $P$ and $P^-$.
Let
$$
S_0=\{\al_1,\dots \al_A, \be_1, \dots \be_B\}
$$
be the simple system for $P_0 \subset P$ ($P_0$ the positive
admissible even system contained in $P$). We denote by
$\al_i$ the compact roots and by $\be_j$ the non-compact (even) roots.

\medskip\noindent
Let us now consider $S$ a simple system for $P$.
Our simple system $S$ is then written as:
$$
S = \{ \al_1,\dots \al_a, \be_1, \dots \be_b, \gamma_1, \dots \gamma_c\}
$$
where the $\gamma_i$'s are simple odd roots, while the number of compact
and non-compact roots may change since we have introduced the odd
roots, 
that is $a \leq A$, $b \leq B$ in general.

As a word of caution let us notice that $S_0$ 
may not be the even part of $S$, since  
in general $S_0 \not\subseteq S$.

\begin{lemma} \label{lem4}
The simple system $S$ contains the same compact roots as $S_0$.
In other words $a=A$.
\end{lemma}

\begin{proof}
Let us assume by contradiction that, say, {$\al_r$} is not in $S$.
Then $\al_r$ is decomposable, so we can write is as $\al_r=\de_1+
\de_2$, where $\de_1$, $\de_2$ are odd (necessarily or otherwise
$\al_{ r}$ would be decomposable in the even positive system).
But the positive system $P$ is admissible, and this is not possible
by the discussion after Definition~\ref{adm-def}. 
\end{proof}

\begin{lemma} \label{lem5}
Let $S_0$ be the simple system for $P_0$ as above.
Then
$$
S^-_0= \{ \al_1,\dots \al_A, -\be'_1, \dots -\be'_B\}
$$
is the simple system for the positive system $P^-_0$ defined by 
\eqref{eq:newpossyst} and $\be_i'$ are the non-compact
roots in $P_{n,0}$ which are the highest weights for the  adjoint  
representation of $\fk$ on $\fp_0^+$ with respect to the positive system $P$.
In particular there exist positive integers $m_i$ such that 
$\be'_i=\be_i+\sum_{\al_i>0} m_i\al_i$.
\end{lemma}

\begin{proof}
This fact is entirely classical.
\end{proof}

We now go to the proof of the main result.

\begin{proof} (Theorem \ref{thm4})
Let us assume by contradiction that $U^\lambda$ has a submodule $M$.
Let $\mu$ be a maximal weight of such submodule, so we may as well
replace $M$ with the cyclic module generated by a weight
vector $v_\mu$. 
{We start by showing that $\lambda-\mu$
is a sum of compact roots.}

\medskip
Since $U^\lambda$ is a highest weight module with respect to $P$,
we have that $\lambda-\mu$ is sum of simple roots of $P$.
On the other hand, by Lemma \ref{lem2} we also have that 
$\lambda - \mu$ is the sum of simple roots of $P_0^-$ {
$$
S^-_0=\{\al_1 \dots \al_A, -\be_1', \dots -\be_B^\prime\}
$$ }
where $\be_i'=\be_i+\sum_{\alpha_i>0} m_i \al_i$ (by Lemma~\ref{lem5}).

\medskip
Hence we can write:
$$
\lambda - \mu= \sum_{i=1}^A a_i \alpha_i+\sum_{j=1}^b 
b_j \beta_j + \sum_{k=1}^cc_k\gamma_k, \qquad a_i, b_j, c_k \in \Z_{\geq 0}
$$
and
$$
\lambda - \mu = \sum_{_i=1}^A a'_i \alpha_i-\sum_{j =1}^B  
b'_j \beta_j  \qquad  b'_j  \in \Z_{\geq 0}
$$
where in the second expression we use $\be'_j$ and then we substitute
its expression in terms of $\be_j$ and compact roots.

\medskip

By 
comparing the two expressions we have:
$$
\sum_{i=1}^A a_i''\al_i + 
\sum_{j=1}^b(b_j+b_j')\beta_j+ \sum_{j=b+1}^B b_j'\be_j+
\sum c_k \gamma_k=0
$$
where the coefficients are all positive, with the exception of the $a''_i$'s.

For $j=b+1, \dots B$ we have:
$$
\be_j=\sum_{i=1}^A m_{ji}\al_i+ \sum_{k=1}^b n_{jk}\be_k+
\sum_{r=1}^c p_{jr}\gamma_r
$$

If we substitute, we get:
$$
\sum_{i=1}^A a_i'''\al_i + 
\sum_{l=1}^b(b_l+b_l'+\sum_{j=b+1}^B b_j'n_{jl})\be_l+
\sum_{k=1}^c (c_k {+}\sum_{j=b+1}^B b_j'p_{jk})\gamma_k=0
$$
where the coefficients are all positive, with the exception of the $a'''_i$'s.

Hence we obtain $b_l=b'_l=c_k=0$, that is $\lambda-\mu$ is the
sum of compact roots.

\medskip
Now we go back to the highest weight vector $v_\mu$ 
of the submodule $M\subseteq\mathcal{U}(\fg)v_\lambda$.
$v_\mu$ is a linear combination of $X_{-\theta_1}\dots X_{-\theta_m}v_\lambda$ 
where each $\theta_j$ is {in $P$}. 
$$
\lambda-\mu=\theta_1+...+\theta_m
$$ 
{Writing each $\theta_j$
as a linear combination of 
$\alpha_i(1\le i\le a)$, $\beta_j(1\le j\le b)$ and $\gamma_k(1 \leq k \leq c)$ 
with integer coefficients $\ge 0$, and noting that 
$\lambda-\mu$ does not involve the $\beta_j$ and $\ga_k$, we conclude that each 
$\theta_j$ does not involve any $\beta_j$ or $\ga_k$. In other words, 
$v_\mu \in \cU(\frak k)v_\lambda$. 
But then $v_\mu$ must be a multiple of $v_\lambda$, 
showing that $M=U^\lambda$.} 
\end{proof}

\subsection{Super Character} \label{char-sec}

In this section we compute the character for the
universal Harish-Chandra module $U^\lambda$ described in the previous sections.

\medskip
Let the notation be as above. If $M$ is an $\fh$-module with finite 
multiplicities, its character is 
given by:
$$
\ch(M)= \sum_{\mu \in \fh^*} (\dim V_\mu) e^{\mu}
$$
 By Corollary \ref{corol::intertwcorol} 
we have  the $\fh$-module isomorphism
$$
U^\lambda \cong \cU(\fp^-) \otimes F
$$
This implies
$$
\ch(U^\lambda)=\ch(\cU(\fp^-)) \ch(F)
$$
We further know that $\cU(\fp^-)\simeq \cU(\fp^-_0) \otimes \wedge(\fp_1^-)$, 
as $\fh$-modules,
hence we can immediately write:
$$
\ch(U^\lambda)=\ch(\cU(\fp^-_0))\, \ch(\wedge(\fp_1^-))\, \ch(F)
$$
Let us quickly recall the following well known expressions 
\begin{align*}
\ch(\cU(\fp^-_0)) & =\prod_{\eta \in P_{n,0}}(1-e^{-\eta})^{-1}=\frac{e^{\rho_{n,0}}}{\Delta_{n,0}}\\
\ch(F) & = \frac{\sum_{s\in W_k}\epsilon(s) e^{s(\lambda+\rho_{k,0})}}
{e^{\rho_{k,0}}\Pi_{\eta\in P_{k,0}}(1-e^{-\eta})}= 
\frac{\sum_{s\in W_k}\epsilon(s) e^{s(\lambda+\rho_{k,0})}}{\Delta_{k,0}}
\end{align*}
where $\epsilon(s)=\det{(s)}$, and  the Weyl denominators are defined as:
\begin{align*}
\Delta_{n,0}&= e^{\rho_{n,0}}\prod_{\eta \in P_{n,0}}(1-e^{-\eta})  \\
\Delta_{k,0}&=e^{\rho_{k,0}}\prod_{\eta \in P_{k,0}}(1-e^{-\eta}) 
\end{align*}
where
$$
\rho_{n,0} = (1/2)\sum_{\al \in P_{n,0}} \al, \qquad
\rho_{k,0} = (1/2)\sum_{\al \in P_{k,0}} \al.
$$
and $P_{n,0}$, $P_{k,0}$ correspond respectively to the non-compact
and compact roots in the (admissible) positive system $P_0$ of $\fg_0$.
Hence:
$$
\ch(F)\ch(\cU(\fp^-_0))=\left( \sum_{s \in W_k} \ep(s)
\frac{e^{s(\lambda + \rho_{k,0})}} {\Delta_{k,0}}\right) \frac{e^{\rho_{n,0}}}
{\Delta_{n,0}}=
\sum_{s \in W_k} \ep(s)\frac{e^{s(\lambda + \rho_{k,0})+\rho_{n,0}}} 
{\Delta_{k,0}\Delta_{n,0}}
$$
(for the classical expression of $\ch(F)$, see, for example, \cite[Ch.4]{vsv1}). Notice that for each $s\in W_k$, $s(\lambda + \rho_{k,0})+\rho_{n,0}=
s(\lambda + \rho_{k,0}+\rho_{n,0})$, hence we write:
$$
\ch(F)\ch(\cU(\fp^-_0))\,=\,\sum_{s \in W_k} \ep(s)
\frac{e^{s(\lambda + \rho_{0})}} {\Delta_{0}}
$$
where $\rho_0=(1/2) \sum_{\al \in P_0} \al$ and $\Delta_0=\Delta_{k,0}\Delta_{n,0}$.

\medskip
We now compute $\ch(\wedge(\fp_1^-))$. 
$$
\ch(\wedge(\fp_1^-))= \prod_{\eta \in P_{1,n}} (1+e^{-\eta})
$$
where $P_{1,n}$ are all the positive non-compact roots.
Notice that in a PBW basis the odd variables appear at most with
degree one. Hence:
$$
\ch(U^\lambda)\, = \,\sum_{s \in W_k} \ep(s)
\frac{e^{s(\lambda + \rho_{0})}} {\Delta_{0}}
\prod_{\eta \in P_{1,n}} (1+e^{-\eta})
$$

\section{Preliminaries on supergeometry}
\label{prelim-sec}

In this section we discuss few facts of supergeometry, 
we need in the following.
We refer the reader for a complete basic treatment of this subject to
\cite{ccf},   \cite{Kostant}, \cite{Leites}, \cite{ma}, \cite{vsv2}.
We are interested
in the {\sl analytic} category of supermanifolds, 
both real and complex, denoted  by
$\smflds_\R$ and $\smflds_\C$ respectively, or by $\smflds$ when
the statement holds in both categories.

\begin{definition} \label{coveringspace}
Let $M=(\red{M},\cO_M)$ and $N=(\red{N},\cO_N)$ be connected 
supermanifolds, i.e. their reduced spaces 
\(\red{M}\) and \(\red{N}\) are connected. 
Suppose $\pi:M\rightarrow N$  is a morphism such that

\begin{enumerate}
\item $\red{\pi}: \red{M}\to \red{N}$ is surjective and is a covering map;
\item for each $x\in \red{N}$ we may choose an  open submanifold 
$\red{U}\subseteq \red{N}$, $x \in \red{U}$, such that 
$\red{\pi}^{-1}(\red{U})=\bigsqcup_i \red{V_i}$, where each \(\red{V}_i \to \red{U}\) is an analytic isomorphism; furthermore
if \(V_i\) is the open subsupermanifold of \(N\) defined by \(\red{V}_i\),
then, for each \(i\), $\pi_{\left.\right|_{V_i}}:V_i \rightarrow U$ 
is an analytic isomorphism;
\end{enumerate}
then we say that $(M,\pi)$ is a {\it covering space} of  $N$.
\end{definition}

\begin{remark}
If $\pi:M\rightarrow N$ is a covering, then
$\red{\pi}:\, \red{M} \rightarrow \red{N}$ is a covering. Moreover, if \(\tau\colon M\to N\) is a morphism, it is a covering map if (and only if) \(\red{\tau}\colon \red{M}\to \red{N}\) is a covering map, \(\dim M=\dim N\) and \(\di \pi\) is  surjective everywhere on \(M\).
\end{remark}

Suppose that $G$ is a SLG. 
We denote with $G_e$ the subsuperLiegroup (subSLG for short) 
of $G$ whose reduced space is 
the identity component of $G$.
(For the relevant notions about quotients see  \cite{bs} and also \cite{ccf}). 

\begin{lemma} \label{lemma2bis}
Suppose $G$ is an 
 analytic connected SLG and $A\subseteq G$ is a closed 
analytic subSLG of 
$G$, $A_e$ its identity component.
Then, in the commutative diagram

$$
\xymatrix{  G  \ar@{->}[rd] \ar@{->}[rr] & & G/A\\ & G/A_e   \ar@{->}[ur] & }
$$

\noindent the map \(G/A_e\to G/A\) is a covering morphism. Moreover
\begin{enumerate}
\item if \(D\) is the discrete even group \(A/A_e\) 
(note that \(A_e\) is open and normal in \(A\)), it acts (from the right) on \(G/A_e\) and commutes with the projection \(G/A_e\to G/A\), acting transitively on the fibers.
\item if we work over \(\R\) and either \(G/A\) or \(G/A_e\) has a complex  
  structure compatible  with the real analytic structure, the other can be equipped with a unique complex structure compatible with the real analytic structure, such that \(G/A_e\to G/A\) is a complex analytic morphism and so a covering map of the complex supermanifolds.
\end{enumerate}
\end{lemma}

\begin{remark}
\label{remarkraja}
It follows from the above that if \(B\) is a subSLG of \(G\) such that \(B_e=A_e\), then the existence of a complex structure on \(G/B\), \(G\)-invariant and compatible with the real analytic one, implies the existence of such a complex structure on \(G/A_e\) and hence on all the \(G/C\) with \(C_e=A_e\). Moreover, if \(B\subset C\), \(G/B \to G/C\) is a covering map for the complex structures.
\end{remark}

\begin{remark} 
If \(M\) and \(N\) are real analytic supermanifolds and \(\pi\colon M\to N\) is a covering map, then a complex structure on \(N\) can be lifted to one on \(M\) in an obvious fashion. In general, to push down a complex structure 
from \(M\) to \(N\) is more complicated, however in the situation considered above, we can do it, since there is a discrete even group acting as a group of super isomorphisms on \(M\), commuting with \(\pi\), acting transitively on the fibers of \(\pi\).
\end{remark}

We end this section with two results on SLG that we shall need
in the sequel.

\smallskip

We first notice that if $M$ and $N$ are supermanifolds 
and $\psi\colon M\to N$ is a submersion, 
then $\red{\psi}\colon \red{M}\to \red{N}$ is an 
open mapping. Hence \(\red{\psi}(\red{M})\) is open in \(\red{N}\) and defines the  open subsupermanifold of 
$N$ given by
\begin{align}
\label{def::opensub}
\psi(M)\coloneqq \left(
\red{\psi}(\red{M}), \restr{\cO_{N}}{\red{\psi}(\red{M})}
\right)
\end{align}

\begin{lemma} 
\label{lemma1}
Let $M$ be a SLG (real or complex), $A_1$ and $A_2$ closed 
subSLG with $\Lie(A_1) + \Lie(A_2)=\Lie(M)$.
Consider the map 
\[
\alpha:\,A_1\times A_2 \stackrel{i_1\times i_2}
{\hookrightarrow} M\times M \stackrel{\mu}{\rightarrow} M
\]
where $i_j:A_i \lra M$ denotes the canonical embedding of $A_i$ in $M$ and
$\mu$ is the multiplication of the supergroup $M$. 
Then
\begin{enumerate}
\item We have \(\red{A}_1 \red{A}_2= \red{\alpha}(\red{A}_1\times \red{A}_2)\) is open in \(\red{M}\) and defines an open subsupermanifold of \(M\), which we write as \(\alpha(A_1\times A_2)\) or \(A_1 A_2\).
\item If \(\red{A}_1 \cap \red{A}_2=\{e\}\)  and \(\textrm{Lie}(A_1) \cap \textrm{Lie}(A_2)=\{0\}\) then \(\alpha\) is an analytic super isomorphism  of \(A_1\times A_2\) with \(A_1 A_2\)
\item \(A_1\) acts transitively on \(A_1A_2/A_2\) and its stabilizer at \(\pi(e)\) (\(\pi\colon A_1A_2 \to A_1A_2/A_2\) is the natural map) is \(A_1\cap A_2\).
\end{enumerate}
\end{lemma}
\begin{proof}
The map, at the functor of points level, is
\[
\alpha_T \colon (a_1,a_2)\mapsto a_1 a_2 \in M(T), \, a_i\in A_i (T)  
\subset M(T)\,,\, A_1(T)A_2(T)\subset M(T), \, T \in \smflds
\]
We first notice that at the topological point $(e,e)\in \red{A_1\times A_2}$
($e$ denoting the identity element),
\[
(\di\alpha)_{(e,e)} (X_1,X_2) =(\di i_1)_e (X_1)+(\di i_2)_e (X_2)
\]
Since $i_1$ and $i_2$ are injective immersions 
and $\Lie(A_1) + \Lie(A_2)=\Lie(M)$, we have that $\alpha$ is a 
submersion at $(e,e)$. For proving that \(\alpha\) is a submersion at any \((\overline{a}_1,\overline{a}_2)\in \red{A}_1\times \red{A}_2\), it is enough to notice that the diagram 

$$
\xymatrix{  A_1\times A_2  \ar@{->}[d]_t  \ar@{->}[r]^\alpha & G\ar@{->}[d]^s\\ A_1\times A_2   \ar@{->}[r]^\alpha & G }
$$

\noindent is commutative, where \(t\) and \(s\) are given by \(t=\ell(\overline{a}_1) \times r(\overline{a}_2)\), \(s=\ell(\overline{a}_1) \circ r(\overline{a}_2)\), \(\ell\,,\,r\) being left and right translations.

For proving 2., note that \(\textrm{Lie}(\red{A}_1)\cap\textrm{Lie}(\red{A}_2)=\{0\}\), hence \(\textrm{Lie}(M)_0=\textrm{Lie}(\red{A}_1)\oplus \textrm{Lie}(\red{A}_2)\). As \(\red{A}_1 \cap\red{A}_2=\{e\} \), \(\red{\alpha}\) is an analytic  isomorphism. Since \(\textrm{Lie}(A_1)+\textrm{Lie}(A_2)=\textrm{Lie}(M)\), and \(\textrm{Lie}(A_1)\cap\textrm{Lie}(A_2)=\{0\}\), \(\dim(A_1\times A_2)=\dim M\). So \(\alpha\) is a covering map. As \(\red{\alpha}\) is an isomorphism, so is \(\alpha\).

Let us come to 3. The statements are clear at the reduced level. Hence, for transitivity of \(A_1\) on \(A_1A_2/A_2\), we must show that \(\di \pi_e (\textrm{Lie}(A_1))=T_{\pi(e)}(A_1A_2/A_2)\). But \(\di\pi_e (\textrm{Lie}(A_1))=\{0\}\) from the theory of quotient spaces, which also gives \(\di \pi_e (\textrm{Lie}(A_1A_2))=T_{\pi(e)}(A_1A_2/A_2)\). Hence, as \(\textrm{Lie}(A_1)+\textrm{Lie}(A_2)=\textrm{Lie}(A_1A_2)\), we must have  \(\di \pi_e (\textrm{Lie}(A_1))=T_{\pi(e)}(A_1A_2/A_2)\).

To find the stabilizer at \(\pi(e)\) of the action of \(A_1\) on \(A_1A_2/A_2\), let \(N\) be the stabilizer. Clearly \(\red{N}=\red{A}_1\cap \red{A}_2\). Since the kernel of \(\di \pi_e\) is \(\textrm{Lie}(A_2)\), the kernel of \(\di \pi'_e\) (where \(\pi'=\restr{\pi}{N}\)) is precisely the space of those vectors in \(T_e (A_2)\) which are tangent to \(A_1\) at \(e\), namely \(T_e (A_1)\cap T_e (A_2)\simeq \textrm{Lie}(A_1)\cap\textrm{Lie}(A_2) \). So \(N=A_1\cap A_2\).

\end{proof}

If $X$ is a complex supermanifold, let us denote with
$X^\R$ its underlying real supermanifold (see \cite{dm, cfk1}).

\begin{lemma}  
\label{lemma2}
Let $G$ be a connected complex matrix Lie supergroup, $\Lie(G)=\fg$.
Let $G_r$ be a connected real form of $G$, $\Lie(G_r)=\fg_r$, 
so that \(\fg \) is the complexification of \(\fg_r\).
Let $R$ be a closed subsupergroup of $G_r$, $\Lie(R)=\fr$. \\
Let $\fq$ be a complex Lie subsuperalgebra of $\fg$ such that
\begin{itemize}
\item $\fg_r+\fq^\R=\fg^\R$;
\item $\fg_r \cap \fq^\R=\fr$.
\end{itemize}
where \(\fq^\R\) and \(\fg^\R\) are the complex Lie superalgebras 
\(\fq\) and \(\fg\) viewed as real  Lie superalgebras.
Assume that the analytic subSLG $Q$ defined by $\fq$ in $G$ is closed.
Let $R_1$ be the SLG with reduced group $\red{Q} \cap \red{G_r}$ and
Lie superalgebra $\fr$. 
\begin{enumerate}
\item $G_r Q^\R$ is an open subsupermanifold of $G^\R$.
Hence, $G_rQ:= (\widetilde{G}\widetilde{Q}, \cO_G|_{\widetilde{G}\widetilde{Q}})$ 
is an open subsupermanifold of the complex supergroup $G$. 
If \(\pi\) is the natural map \(G\to G/Q\), \(M\coloneqq \pi(G_r Q)\) 
is an open subsupermanifold of the complex supermanifold \(G/Q\). 
We also write \(M=G_r Q /Q\).
\item The natural action of \(G_r\) on \(M^\R\) is transitive;
if its stabilizer at \(\pi(e)\) is \(R_1\), then \(G_r/{R_1}\simeq M^\R\). 
Hence
\(G_r/R_1\) acquires naturally a complex supermanifold structure.
\item If \(R'\) is any closed subSLG of \(G_r\) with \(\textrm{Lie}(R')=
{\fr}\), then \(G_r/R'\) acquires a natural  \(G_r\)-invariant complex (super) structure, compatible with the real analytic one. 
\end{enumerate}
\end{lemma}
\begin{proof}

1. By our hypothesis $G_r$ and $Q^\R$ are subsupergroups of $G^\R$ such 
that $\fg_r+\fq^\R=\fg^\R$. Hence by Lemma \ref{lemma1} we are done.
2. Immediate from 3. of Lemma \ref{lemma1}.
3. Follows from Remark \ref{remarkraja}.

\end{proof}

\subsection{Associated Super Vector Bundles and 
Super Fr\'echet Representations}
\label{frechetrep-sec}

For the relevant material about super vector bundles we refer to \cite{dm}. 
Here we recall that a \emph{super vector bundle}
$\cV$ of rank $p|q$ 
on a supermanifold \(M\)
is a locally free sheaf of rank $p|q$ over \(M\),
that is for each $x \in \red{M}$, there exist $U$ open 
such that $\cV(U) \cong \cO_M(U)^{p|q}:= \cO_M(U) \otimes k^{p|q}$.
$\cV$ is a sheaf
of $\cO_M$ modules and at each $x \in \red{M}$, the stalk $\cV_x$ is
a $\cO_{M,x}$-module. We define the \textit{fiber} of $\cV$ at
the point $x$ as the vector superspace $\cV_x/m_{x}\cV_x$, where
$m_x$ is the maximal ideal of $\cO_{M,x}$.  

\medskip

We  briefly recall  the definition of associated super vector 
bundle in the language of SHCP. 
These are super vector bundles on $G/H$
associated to  finite dimensional
$H$-representations, where $H$ is a closed subSLG of $G$.

\begin{definition}
Let $G$ be a SLG, $H$ a closed subSLG, $\sigma$ a finite
finite-dimensional complex representation of $H$ on $V$, {with
$\sigma=(\widetilde{\sigma}, \rho^\sigma)$ in the language of SHCP's.} 
Consider the sheaf over $\red{G}/\red{H}$ 
\[
\cA(U) \coloneqq \cO_G (\red{p}^{-1}(U)) \otimes V, \qquad U \subset \red{G}/\red{H}
\]
where $p:G\rightarrow G/H$ is the canonical projection. We define the super vector bundle associated with \(\sigma\) as
\begin{align}
\label{eq::assbundleSHCP}
U\mapsto \cA_{SHCP}(U)
\end{align}
where:
\begin{align}
\label{eq::sheafSHCP}
\cA_{SHCP}(U) & := \left\{f\in \cA(U)\, |\, 
 \left\{\begin{array}{ll}
( r_h^\ast\otimes 1) f =(1\otimes {\widetilde{\sigma}(h)^{-1})}(f) & 
\forall h\in \wt{H} \\
(D^L_X \otimes 1) f = (1\otimes {\rho^\sigma}{(-X})) f & \forall X\in \hk_1
 \end{array}
 \right.
 \right\}
 \end{align}

\end{definition}

\medskip

One can prove that the previous definition defines a super vector 
bundle over \(G/H\) with typical fiber \(V\).

The following definition is a natural generalization of the one
given in \cite{CCTV}, and it can be found in \cite{Alldridge}. We refer
to \cite{WarnerSemisimple} for the classical result.

\begin{definition} \label{frechetrep-def}
Let $G$ be a complex or real Lie supergroup.
We say we have a  \textit{representation} of $G$ in 
the complex Fr\'echet vector superspace $F$ if:
\begin{itemize}
\item $\wt{G}$ acts on $F=F_0 \oplus F_1$ and:
\[
\red{\pi} \colon \red{G} \to \rAut(F_0) \times \rAut(F_1)
\]
is an ordinary Frech\'et
representation preserving parity.
\item Denote with  $ C\underline{\mathrm{End}}(\cinfty(\red{\pi}))$  
the algebra of continuous linear endomorphisms of  
the space of smooth vectors of $\cinfty(\red{\pi})$
endowed with the Fr\'echet relative topology inherited from $\cinfty{(\red{G};F)}$.
There is an even linear map
\[
\rho^\pi \colon \gk\to C\underline{\mathrm{End}}(\cinfty(\red{\pi}))
\]
such that
\begin{enumerate}
\item $\restr{\rho^\pi}{\gk_0}=\di\red{\pi}$ 
\item $\rho^\pi([X,Y])=\rho^\pi(X)\rho^\pi(Y)-(-1)^{|X|+|Y|}\rho^\pi(Y)\rho^\pi(X)$
\item $\rho^\pi(\textrm{Ad}(g)X)=\red{\pi}(g)\rho^\pi(X)\red{\pi}(g)^{-1}$
\end{enumerate}
\end{itemize}
\end{definition}

We have: 
$F^\infty=F_0^\infty \oplus F_1^\infty$.

It is not difficult to prove the following proposition.

\begin{proposition} \label{action-on-bundle-prop}
Let $G$  and $H$ be as above.
Let $\sigma=({\widetilde{\sigma}, \rho^\sigma)}$ 
be an $H$-representation in the language of SHCP. If $U \subset \red{G/H}$ is a
 $G$-invariant open subset.
The assignment:
\begin{enumerate}
\item
$\wt{G} \times \cA^\sigma(U)  \lra \cA^\sigma(U), \qquad  g,f   
\mapsto l_{g^{-1}}^*f$
\item 
$
\fg  \lra \End(\cA^\sigma(U)), \qquad X  \mapsto D^R_{-X} 
$
\end{enumerate}
gives a representation of $G$ on the Fr\'echet superspace 
$\cA^\sigma(U)$, where 
as usual $D^R_{-X}=(1 \otimes X) \mu^*$ and the element $X \in \fg$ 
is interpreted
as a left invariant vector field.

\end{proposition}

We now turn to examine the decomposition of a
super Fr\'echet representations of a supergroup with respect to
the action of an ordinary compact Lie subgroup.

\medskip

Let $H$ be a real Lie supergroup, and $U$ 
an ordinary compact Lie subgroup in $H$,  $T \subset U$ a maximal torus.
Notice that the following treatment applies also to the case $T=U$.

\medskip

Assume $H$ acts on a Fr\'echet superspace $F$ via the representation
$R$ according to Definition
\ref{frechetrep-def}. Notice that the restriction of the 
representation $R$ to $U$ automatically preserves the $\Z_2$--grading
$F=F_0\oplus F_1$.
 Let $\tau$ be a character of an irreducible representation of $U$, that
we can assume unitary. We define the operator:
\begin{align}
\label{eq::isoproj}
P(\tau)=d(\tau) \int_U \tau(k)^{-1}R(k) \di k, \qquad \hbox{with}
\quad \int_U \di k=1
\end{align}
where $d(\tau)$ is the degree of $\tau$ (namely the dimension
of the irreducible representation associated with $\tau$).
We define $F(\tau)$ as the closed subspace of $F$ 
 stable under $U$ and on which $U$ acts 
according to the irreducible representation with character $\tau$. $F(\tau)$ is called the 
{\it isotypic subspace} corresponding to the character $\tau$.

We stress that, in the whole section, $F^\infty=F^\infty_0 \oplus F^\infty_1$ 
denotes the space of smooth vectors for the representation $R$ of $H$. When we want  to consider smooth vectors for the restriction  of $R$ to a subgroup $U$ of $G$ we will add a subscript $F^\infty_U$.
Clearly, one has $F^\infty\subseteq F^\infty_U$. The following
is a standard result, see \cite{hc2} or \cite[Sections 4.4.2 and 4.4.3]{WarnerSemisimple}.

\begin{proposition} \label{frechetfacts}
Let $R$ be a representation of the compact Lie group $U$ on the Frech\'et 
superspace $F=F_0\oplus F_1$. Then:

\begin{enumerate}
\item \label{item::HC1}
the operator $P(\tau)$ defined by (\ref{eq::isoproj}), is an even continuous projection onto the isotypic subspace $F(\tau)=F(\tau)_0\oplus F(\tau)_1$. 
\item \label{item::HC2}
$F(\tau)$ is a closed subsuperspace of $F$ and it consists of the
algebraic sum of the linear subsuperspaces on   which $U$ acts irreducibly
according to the (irreducible) representation with character $\tau$. Furthermore
the $F(\tau)$ are linearly independent.

\item $P(\tau)P(\tau')=0$, if $\tau \neq \tau'$. 

\item $P(\tau)$ commutes with the $U$ action and with any continuous
endomorphism of $F$ commuting with $U$.

\item On the space of smooth vectors we have $\sum_\tau P(\tau)=\restr{\id}{F^\infty}$, that is any $f \in F^\infty$ is expressed
as $\sum_\tau f_\tau$, which is called the {\sl Fourier series} of $f$.
Furthermore, such series converges uniformly.

\item Let $F^0 :=\sum F(\tau)$ (algebraic sum). Then $F^0 \subset F^\infty$ 
and both are dense in $F$. 

\end{enumerate}

\end{proposition}

When necessary we shall stress the fact that the decomposition of $F$ is under
the $U$-action by writing $F_U^0$ and $F_U(\tau)$, similarly we write
$P_U(\tau)$ for the operator defined in (\ref{eq::isoproj}).



\begin{definition}
We say that a representation $R$ as above is \textit{$U$-finite} if
every $F(\tau)$ is finite dimensional. 
\end{definition}

The following is a standard lemma, that we leave to the reader
as an exercise.

\begin{lemma}  
\label{finiteness-cor}
Let the notation and setting be as above.
\begin{enumerate}
\item Let $\widehat{F}^0=\sum_\tau L_\tau$ 
be a dense subspace in $F$, where
the sum is algebraic, the subspaces $L_\tau$ are all
finite dimensional and $L_\tau \subset F_U(\tau)$. 
Then $L_\tau=F_U(\tau)$ for all $\tau$ and
$\widehat{F}^0=F^0_U \subset F^\infty$. Hence, $F$ is $U$-finite.   
\item Let {$U^\prime$} be a compact subgroup of $U$ and assume that $F$
is {$U^\prime$}-finite. Then $F$ is also $U$-finite and $F^0_U=F^0_{{U^\prime}}$.
\end{enumerate}
\end{lemma}

\section{Representations of the Supergroup}\label{grouprep-sec}

The objective of this section is to construct representations
of a real supergroup $G_r$ which correspond infinitesimally to the highest
weight Harish-Chandra modules. 

\medskip
Let $\fg$ be as in list \eqref{list:SLA} and let 
$\fg_r$ be a real form of $\fg$
(see \cite{parker, cf}).
By the ordinary theory, we know that, since
$\fg_0$ is either semisimple or with a one-dimensional center,  
the simply connected corresponding ordinary Lie group 
$\wt{G}$ is a matrix Lie and algebraic group. 
Then, the SHCP $G=(\wt{G}, \fg)$ (see \cite{ccf} Ch. 11 and \cite{fi2})
can be viewed either as an algebraic or an analytic complex SHCP. 
Hence $G$ is a complex analytic matrix supergroup and 
$G^\R$, the supergroup $G$ viewed as a real supergroup (see
\cite{dm, cfk1}), is also
a real analytic matrix supergroup.
Let $G_r$ be the real analytic supergroup corresponding to
the real subsuperalgebra $\fg_r$ of $\fg^\R$ (the superalgebra $\fg$
viewed as real superalgebra). Also $G_r$ is a matrix real Lie supergroup
and we will refer to $G$ as the complexification of $G_r$ and 
we will refer to $G_r$ as a real form of $G$.

\medskip
Fix $\fh$ and $\fh_r$ CSA of
$\fg$ and $\fg_r$ respectively, $\fh$ the complexification of $\fh_r$.
$K_r=\wt{K_r}$ 
is the maximal compact in $\wt{G_{r}}$, $A_r=\wt{A_r}$
the (ordinary) torus, $A_r \subset G_r$, while $\fk_r$, $\fh_r$
the respective Lie superalgebras. We drop the index $r$ to mean the
complexifications. We assume:
$$
\fh \subset \fk \subset \fg, \qquad \fh_r \subset \fk_r \subset \fg_r.
$$
Hence our CSA $\fh=\fh_0$. 
Let $\Delta$ be the root system corresponding to $(\fg,\fh)$
and fix $P$ a positive system.
Let us define $\fb^\pm$ and $\fn^\pm$ 
the \textit{Borel and nilpotent subsuperalgebras}:
$$
\fg= \fh \oplus \bigoplus_{\al \in \Delta} \fg_\al, \qquad
\fb^\pm:= \fh \oplus \sum_{\al \in \pm P} \fg_\al, \qquad
\fn^\pm:=  \sum_{\al \in \pm P} \fg_\al.
$$
We  call
$B^\pm$ \textit{Borel subsupergroup} and 
$N^\pm$ \textit{unipotent subsupergroup}, their
corresponding analytic Lie supergroups in $G$.
In particular, $B^\pm$and $N^\pm$ are connected and are algebraic
subsupergroups of $G$.

\subsection{Maximal torus and big cell in Lie supergroups
of classical type}
\label{maxtorus-sec}

In this section we want to study the connected
ordinary Lie group $A \subset G$, called a \textit{maximal torus} of $G$,
with associated Lie algebra $\Lie(A)=\fh$ the CSA of $\fg$, and
its relation with the supergroups $N^\pm$. 
In particular we introduce the supermanifold $\Gamma:=N^-AN^+ \subset G$,
called the \textit{big cell},
which plays a key role in what follows.
We observe first that 
the (ordinary) torus $A$ normalizes $N^\pm$, as it happens for the 
ordinary setting.

\begin{proposition} \label{lemma3}
Let $\fm_r$ be a real form of the Lie superalgebra $\fg=\Lie(G)$ 
containing $\fh_r$ the CSA of $\fg_r$. Then:
$\fm_r + \fb^+=\fg$ and in particular $M_r(B^+)^\R$ is open subsupermanifold 
of $G^\R$, 
where $M_r$ is the connected subSLG of $G_r$ determined by $\fm_r$.
\end{proposition}
\begin{proof}
Since $\fm_r$ is a real form of $\fg$, we have $\fg=\fm_r\oplus i\fm_r$. This is equivalent 
to say that there exists an antilinear involution $\widetilde{\,\,}\colon \fg\to \fg$ whose set of fixed points is $\fm_r$.
 Moreover, since $\fh_r$ is contained in $\fk_r$ we have that all the roots are imaginary when restricted to $\fh_r$. These facts  imply that $\fg_\alpha\widetilde{}=\fg_{-\alpha}$.

In order to prove our statement it is enough to show that 
$X_{-\al}$ and $iX_{-\al}$ belong to $\fm_r+\fn^+$.
We have that:
\begin{align*}
X_{-\al}&=X_\al\widetilde{}=(X_{\al}+{X_\al}\widetilde{}~)-
X_\al \in \fm_r+\fn^+ \\ 
iX_{-\al}&=-iX_\al\widetilde{}=(-iX_{\al}\widetilde{}+iX_\al)-iX_\al
 \in \fm_r+\fn^+
\end{align*}
Hence, by Lemma \ref{lemma1}, we obtain our result.
\end{proof}

\begin{proposition} \label{lemma4}
Let the notation be as above. 
Then we have that: 
\begin{enumerate}
\item $\red{\Gamma}:=\red{N^-AN^+}$ is open in $\red{G}$. \\
\item $\red{A}$, $\red{N^\pm}$ are closed and
$\red{N^\pm} \cap \red{A}  =\{1\}$. \\ 
\item The morphism
$$
N^- \times A \times N^+ \lra G, \quad (n^-,h,n^+) \mapsto n^-hn^+,
\quad n \in N^\pm(T), \quad h \in A(T), \, T \in \smflds_\C
$$
is an analytic isomorphism onto its image $N^-AN^+$ which is
an open subsupermanifold of $G$. 
\end{enumerate}
\end{proposition}
 
\begin{proof} 
(1) and (2) are statements of ordinary geometry. (3)
Consider the morphism $\phi:A \times N^+ \lra AN^+ \subset G$. 
$AN^+$ is a Lie supergroup,
since $N^+$ is normalized by $A$  and $\Lie(AN^+)=\fh+\fn^+$. Hence
$\phi$ is a diffeomorphism onto its image (apply Lemma \ref{lemma1}).  
We now apply again Lemma \ref{lemma1} to the  map $\psi:N^-\times AN^+ \to G$.
$\psi$ is a diffeomorphism onto its image, which is an open subsupermanifold of $G$. 
\end{proof}

\begin{remark} 
The image of
the multiplication morphism $A \times N^+ \lra AN^+ \subset G$ 
is a Lie supergroup.
Since its reduced space is $\red{B^+}=\red{AN^+}$ and
its Lie superalgebra $\fb^+= \fh \oplus \fn^-$, we have that $B^+=AN^+$
(and similarly $B^-=AN^-$), where $B^+$ is the unique connected subsupergroup of
$G$ with Lie superalgebra $\fb^+$.
\end{remark}

\begin{definition}
Let the notation be as above. 
We define the \textit{big cell} in $G$ as 
the open subsupermanifold of $G$:
$$
\Gamma:=N^-A N^+ \subset G
$$
Its underlying topological space
$\red{\Gamma}=\red{N^-AN^+}$ is open and dense in $G$.
\end{definition}

\begin{proposition} \label{lemma5}
Let the notation be as above.
Then we have that:
\begin{enumerate}
\item  $G_r(B^{\pm})^\R$ are open real subsupermanifolds in $G^\R$; 
$G_rB^{\pm}=(\wt{G_r}\wt{B^\pm}, \cO_G|_{\wt{G_r}\wt{B^\pm}})$ are 
open complex subsupermanifolds in $G$;\\
\item  $G_r/A_r \cong G_rB^\pm/B^\pm$ acquires
a {$\wt{G_r}$ invariant complex
structure.} \\
\item  $N^-$ is a section for $\Gamma\to \Gamma/B^+$, 
the left action of $A$ reads:
$$
A \times \Gamma/B^+ \lra \Gamma/B^+, \, \,
(h, nB^+(T)) \mapsto hnh^{-1}B^+(T), \,\, n \in N^\pm(T), \, 
h \in A(T), \, T \in \smflds_\C
$$
\end{enumerate}
\end{proposition}
\begin{proof}
(1) Due to Lemma~\ref{lemma1}, $\fg_r+\fb^\pm=\fg$, hence the map 
$\alpha\colon G_r\times B^\pm\to G$ is a subsupermersion and
the subsupermanifolds $G_r(B^\pm)^\R$ are open in $G^\R$.

(2) If we prove that $\fg_r\cap\fb^\pm=\fh_r$ and $\red{G_r}\cap \red{B^+}=\red{A_r}$ we can apply Lemma~\ref{lemma2} and conclude. Let us hence proceed to prove these facts.
We have 
\[
\fg_r\cap \fb^\pm=\fh_r\,.
\]
Indeed, let $X\mapsto X\widetilde{\,}$ be the conjugation of $\fg$ associated with $\fg_r$ as described in the proof of Prop~\ref{lemma3}. Consider the case of $\fb^+$ for definiteness. $X\in\fb^+$ can be written as 
$X=\sum c_i H_i +\sum_{\alpha \in \Delta^+} d_\alpha X_\alpha$ with 
$H_i\in\fh_r$, $X_\alpha\in\fg_\alpha$, and $c_i,d_\alpha\in \C$. Then
$X\widetilde{\,}=\sum_i \overline{c}_i H_i +
\sum_{\alpha \in \Delta^+} \overline{d}_\alpha X_{\alpha}\widetilde{\,}$.
Since $\fg_\alpha\tilde{\,}=\fg_{-\alpha}$ (see the proof of Prop~\ref{lemma3}), we have that $X\widetilde{\,}=X$ if and only if $X\in\fh_r$.\\ Since by assumption  the SLG $G$ is simply connected, there exists an antiautomorphism $\sigma\colon G \to G$ such that $(\di\sigma)_e(X)=X\widetilde{\,}$. 
By the ordinary theory we have that 
$\red{G_r}\cap \red{B^+}=\red{A_r}$. Indeed, let 
  $a=hn^+ \in \red{G_r}\cap \red{B^+}$, with $h \in \red{A_r}$ and
$n^+ \in \red{B^+}$.
Hence $hn^+=\red{\sigma}(h)n^-$ (where $n^{-}\coloneqq \red{\sigma}(n^+)$), that is $\red{\sigma}(h)^{-1}hn^+=n^-$, so that
$h=\red{\sigma}(h)$, $n^+=n^-=1$, so $a \in \red{A_r}$.

(3) Since the big cell $\Gamma\subset G$ is right $B^+$-invariant and open, 
and the canonical projection $p\colon G\to G/B^+$ is a submersion,  
we can define the open subsupermanifold of $G/B^+$:
\[
{\Gamma}/B^+ \coloneqq (\red{\Gamma/B^+}, \restr{\cO_{G/B^+}}
{\red{\Gamma/B^+}})
\]
We have a $N^-$ equivariant diffeomorphism $N^- \lra \Gamma/B^+$,
$n^- \mapsto n^-B^+(T)$, $n^- \in N^-(T)$,  $T \in \smflds_\C$. 
In fact, by the ordinary
theory we have a diffeomorphisms of the underlying differentiable manifolds
and the differential at the identity is an isomorphism: $\fn^- \cong \fg/\fb^+$.
\end{proof}

\subsection{Line bundles on $G/B^+$} \label{linebu-sec}

Let us consider a character $\chi_r$ of the classical real maximal torus $A_r$
inside the real supergroup $G_r$. This character uniquely extends to an
holomorphic character of $A$ and has the form 
\begin{align*}
\chi:  A & \lra  \C^\times \\ 
 exp(H) & \mapsto  e^{\lambda(H)}
\end{align*}
for an integral weight $\lambda \in \fh^*$ (i.e. a weight
such that $\lambda(H_\gamma) \in \bZ$ for all roots $\gamma$).

We can 
trivially extend the character $\chi$ of $A$ to a
character of the Borel subsupergroup $B^+$, since we know
$B^+=AN^+$. 
We denote with {$(\chi_0, \lambda)$} the corresponding representation in
the SHCP formalism.

\medskip
The character $\chi=e^\lambda$ defines according to \ref{frechetrep-sec}
an {holomorphic} 
line bundle on $\red{G/B^+}$ that we denote with $L^\chi$ or 
$L_\lambda$ depending on the convenience.
If $p:{G} \lra {G/B^+}$ we have:
\beq
\label{def::Lchi}
L^\chi(U) =L_\lambda(U)=\left\{ 
f\in \cO_G(p^{-1}(U))\, |\,  
\begin{cases}
 r_b^\ast f =\chi_0(b)^{-1} (f) 
& \forall b\in \red{B^+}\\
  D^L_X  f = \lambda{(-X}) f &  \forall X\in \fb^+
 \end{cases}
\right\} 
\eeq

\noindent We can equivalently write:
$$
L^\chi(U)=L_\lambda(U)=\left\{ f:  p^{-1}(U)\rightarrow 
{\C}^{1 \mid  1} \, 
|\right.   \,  f_T(gb)=\chi_T(b)^{-1} f_T(g),
\left. \, b\in B^+ (T)\,,\, g\in U(T)\right\}
$$

\medskip
We now turn our attention to the  Frech\'et superspace
$F:=L^\chi(\red{\Gamma/B^+})$, where $\Gamma=N^-AN^+$ is the big cell
in the complex supergroup $G$. $\Gamma$ is neither stable under
$G$-action nor under the $G_r$-action, 
however, as any neighbourhood of the identity, it is stable
under the action of $\cU(\fg)$ and we want
to study such representation.

\begin{proposition} \label{gamma-prop}
The restriction of the holomorphic line 
bundle $L^\chi$ to  $\red{\Gamma/B^+}=\red{N^- A N^+/B^+}$ is trivial:
\[
L^\chi(\red{\Gamma/B^+}) \simeq \cO_{G/B^+}(\red{\Gamma/B^+})
\]
In particular, there is a canonical identification between:
\begin{align}
\label{eq::topoident}
F=L^\chi(\red{\Gamma/B^+}) \simeq  \cO(N^-)
\end{align}
between the sections on the big cell of the line bundle $L^\chi$
and the holomorphic functions on $N^-$.
\end{proposition}

\begin{proof} 
In order to prove the triviality
of the line bundle $L^\chi$ over 
$\red{\Gamma/B^+}$, we have to construct a section
\(
s\colon {\Gamma}/B^+ \to \Gamma
\).
This is the content of Prop. \ref{lemma5}, (3). 
The  isomorphism  \eqref{eq::topoident} is easily established 
using the correspondence between  sections of the associated bundle \(L^\chi\) and the \(B^+\)-equivariant mappings \(N^- B^+ \to \C^{1|1}\), as in the classical setting (see, for example, \cite{mi}). More precisely, let
\[
\kappa\colon N^-\times B^+ \stackrel{\simeq}{\to} N^- B^+
\]
be the isomorphism established in Prop.~\ref{lemma4}.
We have maps $\eta \colon L^\chi(\red{\Gamma/B^+})\to \cO(N^-)$ 
and $\zeta\colon \cO(N^-)\to L^\chi(\red{\Gamma/B^+})$ given by
\begin{align*}
\eta\colon L^\chi(\red{\Gamma/B^+}) & \to \cO(N^-) & \zeta\colon  \cO(N^-) & \to L^\chi(\red{\Gamma/B^+})\\
f & \to G_f\coloneqq \eta(f) & G & \to f_G\coloneqq \zeta{(G)}
\end{align*}
where $G_f$ and $f_G$ are the morphisms defined as follows
\[
G_f\colon N^- \stackrel{i}{\to} N^-\times B^+ \stackrel{\kappa}{\to} 
N^-B^+\stackrel{f}{\to} \C^{1|1}
\]
and
\[
f_G\colon N^-B^+ \stackrel{\kappa^{-1}}{\to} N^-\times B^+ 
\stackrel{G\times \chi}{\longrightarrow} \C^{1|1}\times 
\C^\times 
\to \C^{1|1}
\]
$\eta$ and $\zeta$ gives the desired isomorphism, 
we leave to the reader the standard checks involved. 
\end{proof}

\begin{remark} 
Since the Frech\'et topology on \(L^\chi\) is defined through local trivializations, by the previous Proposition we have that 
the identification $L^\chi(\red{\Gamma/B^+})\simeq \cO(N^-)$  is also 
an isomorphism of Frech\'et superspaces, an in the ordinary case.
\end{remark}

To ease the notation we shall also write $L^\chi({\Gamma})$
in place of $L^\chi(\red{\Gamma/B^+})$.

\medskip

Let $t_\al$ denote the global homogeneous
exponential coordinates on $N^-$  (see \cite{koszul,gw} for
the definition of exponential). By a classical result, if 
$\red{N}$ is a connected nilpotent Lie 
group, $\cU(\fn_0)$ preserves the 
ordinary polynomials $\cP(\red{N})$ on $\red{N}$.
Let $\cP=\cP(\red{N}^-) \otimes \wedge(\fn_1^-)$.
We thus have the natural identifications:
\[
\cP=\cP(\red{N}^-) \otimes \wedge(\fn_1^-)=
\uHom_{\fn_0}(\cU(\fn^-),\cP(\red{N}^-))
\subset  \uHom_{\fn_0}(\cU(\fn^-),C^\infty(\red{N}^-))
\]
Notice that  $\cP$ are the polynomials in the
indeterminates $t_\al$.

We now want to study in detail the action of $A_r$, the ordinary
torus in $G_r$, on the polynomials $\cP$ in $\cO(N^-)$ 
and the corresponding
superalgebra $\cov{\cP}$ in $F=L^\chi({\Gamma})$.

\medskip

Proposition \ref{gamma-prop} allows us to
obtain immediately the following corollary 
(we shall also see it later as a consequence of 
Lemma~\ref{spectrum-lemma} in the next section).

\begin{proposition} 
$\cP$ is dense in $\cO(N^-)$ and $\cov{\cP}$ is dense in $F$.
\end{proposition}
\begin{proof}
{In view of the definition of the topology on $F$, 
it is enough to prove that $\cP$ is dense in $\cO(N^-)$. 
The proof goes as in the ordinary setting, since
$N^-$ is analytically isomorphic to $\C^{m|n}$ via 
the exponential morphism, for
suitable $m|n$.}
\end{proof}

\subsection{The action of the maximal torus on the 
polynomials on the big cell} \label{actiontorus-sec}

In this section we introduce two natural actions $c$ and $l$ of 
the ordinary Lie group $A_r$ on the big cell $\Gamma=N^- A N^+$, 
together with the actions $i$ and $\ell$ they induce on
the Frech\'et superspace $L^\chi(\Gamma)$. We notice for future reference that both actions coincide on the quotient \(\Gamma/B^+\).

\medskip

In the definition of the two actions we use the 
isomorphism (see Prop.~\ref{lemma4})
\[
\kappa\colon N^-\times B^+ \stackrel{\simeq}{\to} \Gamma
\]

Let us start with the action $c$ related to the coniugation.
Since $A_r$ acts on $N^-$ by conjugation (see Prop.~\ref{lemma5}),  
we have a global action of $A_r$ on $\Gamma$ defined as: 
\begin{align}
\label{eq::actionAr1}
c\colon A_r \times \Gamma  \stackrel{1_{A_r}\times \kappa^{-1}}
 {\to}  A_r\times (N^-\times B^+)  \stackrel{\textrm{conj}\times 1_{B^+}}
 {\to}  N^-\times B^+ \stackrel{\kappa}  {\to}  \Gamma
\end{align}
which in the functor of points notation reads 
\[
a\cdot (n^- b^+)=(a n^- a^{-1}) b^+,\quad
a \in \wt{A_r}, \, n^- \in N^-(T), b^+ \in B^+(T).
\]
Since $A_r$ also acts on $B^+$ by left translation $l^\prime$, 
we can define the left action of $A_r$ on $\Gamma$ as
\begin{align*}
l_a & = \kappa\circ(\textrm{conj}_a\times l^\prime_a)\circ\kappa^{-1}
\end{align*}
or, in the functor of points notation,
\begin{align}
\label{eq::actionAr2}
a\cdot (n^- b^+)=(a n^- a^{-1})a\cdot b^+.
\end{align}

Both actions commute with right translations by $B^+$ and hence 
define  representations of $A_r$ on $L^\chi(\Gamma)$
\begin{align*}
i,\, \ell \colon A_r \times L^\chi(\Gamma) & \to L^\chi(\Gamma) 
\end{align*}
where \(
i_a(f)  =c_{a^{-1}}^\ast (f)\) and \(
\ell_a(f)  = l_{a^{-1}}^\ast(f)
\)
for all $a\in \red{A_r}$, and all $f\in L^\chi (\Gamma)$.

These representations are most easily written in 
the functor of points notation as
\begin{align*}
i_a (f) (n^- b^+) & = f((a^{-1}n^- a)b^+)\\
\ell_a (f) (n^- b^+ ) & = f((a^{-1}n^- a) a^{-1}b^+)
\end{align*}
The above formulas further simplify using the identification 
\[
L^\chi(\Gamma)\simeq \cO(N^{-})\,,
\]
we leave the details to the reader.

\begin{lemma} \label{torus-actions}
Let the notation be as above. Then
\begin{enumerate}
\item $\ell_a f=\chi(a) (i_af)$ 
\item $i_at_\al=\chi_\al(a)t_\al \qquad \forall a \in \wt{A_r}$
\end{enumerate}
\end{lemma}

\begin{proof}
(1) follows  immediately from (\ref{eq::actionAr2}). For (2) let
\(
n=\exp(\sum_{\beta\in P} y_\beta X_{-\beta})
\) 
in $N^-(T)$, then the result 
comes from the following formal calculation:
\begin{align*}
t_\al(a^{-1}na)&=t_\al\big(\exp(\sum_{\beta\in P} y_\beta Ad(a)X_{-\beta}) \big)=
t_\alpha(\exp(\sum_{\beta\in P} y_\beta \chi_\beta(a)X_{-\beta}) \\
&=\chi_\al(a)t_\al(n), 
\end{align*}
where $a \in \wt{A_r}$, $y_\beta \in \C$ {and the $t_\al$ are the polynomial
coordinates on $N^-$ (see Sec. \ref{linebu-sec}).}
\end{proof}

To ease the notation we shall also write $a \cdot f$ in place of
$\ell_a(f)$.

\begin{proposition}  
Let $\cP$ be the polynomial superalgebra generated by the $t_\al$ in
$\cO(N^-)$ and let $\cov{\cP}$ be the corresponding submodule in $F$. 
$A_r$ acts on $\cov{\cP}$ and we have that:
$$
a \cdot \cov{(t^{r_{\al_1}}_{\al_1} \dots t^{r_{\al_s}}_{\al_s})} = 
\chi_{\lambda+\sum r_{\al_i}\al_i}(a) \,  
\cov{(t^{r_{\al_1}}_{\al_1} \dots t^{r_{\al_s}}_{\al_s})}
$$
Hence $\cov{\cP}$ decomposes into the sum of eigenspaces 
$\cov{\cP}_d$ 
for the action of $A_r$, where $d$ ranges in $D^+$
the semigroup in $\fh^*$ generated by the positive roots:
$$
\cov{\cP}=\oplus_{d \in D^+}\cov{\cP}_d, \qquad 
\cov{\cP}_d=\oplus_{\sum r_{\al_i}\al_i=d} \, \C \cdot 
\cov{(t^{r_{\al_1}}_{\al_1} \dots t^{r_{\al_s}}_{\al_s})} 
$$
A similar decomposition holds also for $\cP$.
\end{proposition}

\begin{proof} This is a simple calculation, similar to the one
in Lemma \ref{torus-actions}.
\end{proof}

\begin{corollary}
\label{cor::dimspectrum}
The maximal torus $A_r$ acts on the  Frech\'et superspace $L^\chi(\Gamma)$
and we have: 
\begin{enumerate} 
\item $F(\tau) \neq 0$ if and only if $\tau=\chi_{-\lambda+d}$ for
some $d=\sum_{m_\al \in \bZ_{\geq 0}, \al \in P} m_\al \al$.
\item
\[
F(\chi_{\lambda+d})=\cov{\cP}_{\lambda+d}
\]
and 
\[
\dim(F(\chi_{\lambda+d}))=
\# \left\{r=(r_\alpha) \, | \, \sum_{r_\al \in \bZ_{\geq 0}, \al \in P} r_\al \al=d\right\}
\]
\end{enumerate}
\end{corollary}
\begin{proof}
(1) and (2) are consequences of  Lemma~\ref{finiteness-cor}. 
The computation of the dimension is straightforward.
\end{proof}

We now want to prove the fact that the spectrum of $A_r$ 
remains unchanged when
we change the open set we are considering in a suitable way. We shall first
prove a general lemma. 

\medskip

Let $T$ be an ordinary compact torus acting on a finite dimensional 
complex vector superspace $V$. 
By a classical result we have  
the action of $T$ on $V$ is via characters $\tau_i$'s and, 
with a suitable choice of a basis of \(V\), reads as follows
\begin{align*}
T \times V & \lra  V \\ 
t, (v_1, \dots, v_m, \nu_1, \dots, \nu_n) & \mapsto  (\tau_1(t)v_1, \dots
\tau_m(t)v_m, \tau_{m+1}(t)\nu_1, \dots , \tau_{m+n}(t)\nu_n)
\end{align*}
We can easily transport this action to the space of polynomial functions
$\Pol(V)$ on $V$ and obtain the following action:
$$
t \cdot \sum a_{IJ} z^I\xi^J = \sum a_{IJ} \tau^I (t)^{-1}
\tau^J (t)^{-1}z^I \xi^J
$$ 
using the multiindex notation $I=(i_1 \dots, i_n)$ with possibly repeated
indices, $J=(j_1, \dots , j_n)$ with no repeated indices.

$T$ has also a natural action on
the holomorphic sections of the structural sheaf on $V$, $\cO_V(U)$,
where $U$ is a $T$ invariant open set in $V$. 
If $g \in \cO_V(U)$, we know we can view such $g$ as 
a morphism $g: U \lra \C^{1|1}$.
If $\rho_t(u):=t \cdot u$ is the action of $T$ on $U$, we define:
$$
t \cdot g = g \circ \rho_{t^{-1}}
$$
Notice that such action agrees with
the previously defined action on the polynomials.

\medskip

We define $\Pol(\tau)$ the space of polynomials transforming according
to the character $\tau$, that is:
$$
\Pol(\tau)=\{p\in \Pol(V) \, \big| \, t \cdot p=\tau(t) p\}
$$   
We define also $\Pol(U)=\Pol(V)|_U$, for any open $U \subset V$.

\begin{lemma} \label{spectrum-lemma}
Let $T$ be an ordinary compact torus acting on 
 a finite dimensional complex vector superspace $V$. 
For any character $\tau$ of $T$,
we assume that $\dim(\mathrm{Pol}(\tau))<\infty$. Then any
open connected subset $U$ of $V$ which is $T$-invariant and contains the origin,
is such that $\mathrm{Pol}(U)$ is dense in $\cO_V(U)$.
\end{lemma}

\begin{proof} 
We may assume that $V = \C^{m|n}$ with $T$-action
$$
t, (z_1, . . . , \xi_{m+n}) \mapsto 
(f_1(t)z_1, . . . , f_{m+n}(t)\xi_{m+n}) 
$$
where the $f_j$ are characters of
$T$. 
Let $U$ be an open connected subset of $\C^{m|n}$ containing the origin and
stable under $T$. The action of $T$ induces an action on $\cO_V(U)$. 
It is enough to prove that the closure of $\Pol(\C^{m|n})$ 
contains $\cO_V(U)^\infty$ the smooth vectors in $\cO_V(U)$ with respect to
the $T$ action, since we know such space is dense in $\cO_V(U)$. 
Since the Fourier series of any $g$ in $\cO_V(U)$ converges to $g$
(see Prop. \ref{frechetfacts} (5) and (6)), it
is enough to show that any eigenfunction of $T$ in $\cO_V(U)$ is a polynomial.
Suppose $g \neq 0$ is in $\cO_V(U)$ 
such that $t^{-1}\cdot g = f(t)g$ for all $t \in T$ and
$u \in U$, $f$ being a character of $T$. 
Since $0 \in U$ we can expand $g$ as a power
series $g(u) = \sum c_ru^r$  in a polydisk, where $u$ comprehends even and
odd coordinates and we are using the multiindex notation. Notice that
the action of $T$ preserves the polidisks, and we have
\[
(t^{-1}.g)(u)=g(tu)=\sum_r c_r (t u)^r=\sum_r c_r f^r(t) u^r=f(t)g=
\sum_r c_r f(t) u^r
\]
Then $c_rf^r = c_rf$ whenever $c_r \neq 0$, because $t^{-1}\cdot g = f(t)g$.
So only the $r$
with $f = f^r$ appear in the expansion of $g$. We claim that there are only
finitely many such $r$; once this claim is proven we are done, because $g$
is a linear combination of the monomials $u^r$ with $f^r = f$, hence $g$ is a
polynomial. To prove the claim, note that all such $u^r$ are eigenfunctions
for $T$ for the eigencharacter $f$, and by assumption, there are only finitely
many of these.
\end{proof}

We want to apply the previous lemma in a case that is of interest to us.

\medskip

Define now $\Gamma_1=G_rB^+/B^+$ and $\Gamma_2=(\Gamma \cap G_r B^+)^0/B^+$
{(the suffix ``$0$'' denotes the connected component of the identity).}
These are open sets in $G/B^+$ which are invariant under the $A_r$ action.
Let us denote with $\cov{\cP}(\Gamma_2)$, the set $\cov{(\cP|_{\Gamma_2'})}$
where $\Gamma_2' \cong \Gamma_2$ in the isomorphism of analytic supermanifolds
$N^- \cong \Gamma/B^+$.

\begin{corollary}  \label{spectrum-cor}
Let the notation be as above.
\begin{enumerate}
\item $\overline{\cP}=\cO(N^-)$, 
where the nilpotent supergroup $N^-$
is interpreted as a vector superspace via the identification $\fn^- \cong N^-$
via the exponential morphism.
\item $\overline{\cov{\cP}}=L^\chi(\Gamma)$, $\Gamma=N^-AN^+$ the
big cell in $G$.
\item 
$\overline{\cov{\cP}}(\Gamma_2)=L^\chi(\Gamma_2)$.
\end{enumerate}
\end{corollary}
\begin{proof}
(1) We apply Lemma~\ref{spectrum-lemma}. The torus $A_r$ acts on $N^-$ through 
the action $c$ given by (\ref{eq::actionAr1}).
 The condition $\dim\Pol(\tau)<\infty$ is 
checked with a calculation completely similar to that 
of Corollary~\ref{cor::dimspectrum}.
(2) is a consequence of the  isomorphism (\ref{eq::topoident}).
(3) follows again from Lemma~\ref{spectrum-lemma}.
\end{proof}

Define now $F=L^\chi(\Gamma)$, $F^1=L^\chi(\Gamma_1)$,
$F^2=L^\chi(\Gamma_2)$. Notice that on $F$ and $F^2$ we do not have any $G$ or
$G_r$ action, only $F^1$ is a $G_r$ module in a natural way.

\begin{corollary}\label{spectrum-cor2}
Let the notation be as above.
\begin{enumerate}
\item The restriction morphism $F^1 \lra F^2$ is a continuous injection.
\item Under the restriction, $F^1(\tau) \subset F^2(\tau)$ for
characters $\tau$ of $A_r$.
\item $F^2(\tau)=F(\tau)|_{\Gamma_2}$.
\end{enumerate}
\end{corollary}

\begin{proof}
$(1)$ and $(2)$ are clear, $(1)$ because $\Gamma_2$ is open in $\Gamma_1$  
and of the  analytic
continuation principle, which holds also in the supersetting,
while $(2)$ is a simple check. Now we go to $(3)$. The space of polynomials
$\Pol(\Gamma)$ on $\Gamma$ is dense in $F$ and by 
Corollary~\ref{cor::dimspectrum} we have 
$F^0=\Pol(\Gamma)$.  $\Pol(\Gamma_2)=\Pol(\Gamma)|_{\Gamma_2}$
is dense in $F^2$ by the previous corollary. Since $\Pol(\Gamma_2)$ is dense
in $F^2$, 
we have that  the restriction of $F(\tau)$ to $\Gamma_2$ is 
dense in $F^2(\tau)$ and since $F(\tau)$
is finite dimensional we have $F(\tau)|_{\Gamma_2}= F^2(\tau)$.
\end{proof}

\subsection{The action of $\cU(\fg)$ on $L^\chi(\Gamma)$}
\label{actionU-sec} 

We start by defining the natural action of $\cU(\fg)$ on the
holomorphic functions on any neighbourhood $W$ of the identity
of the supergroup $G$.

\begin{definition}
Let $W \subset G$ be an open neighbourhood of the identity $1_G$ in $G$.
There are two well defined actions of $\fg$, hence 
of $\cU(\fg)$, on $\cO(W)$ that read as follows:
$$
\ell(X)f=(-X \otimes 1)\mu^*(f) , \qquad X \in \fg
$$
$$
\partial(X)f=(1 \otimes X)\mu^*(f)
$$
\end{definition}

\begin{proposition}
Let 
$U$ be open in $\red{G/B^+}$.
Then  $\ell$ and $\partial$ are well defined actions
on $\cO(U)$ and they commute with each other.
\end{proposition}

\begin{proof} Immediate. \end{proof} 

The natural action of
$\cU(\fg)$ on $L^\chi(N^-B^+)$ is algebraic,
hence it preserves $\cov{\cP}$. The proof is analogous to the classical one and we leave the details to the reader.

\begin{proposition} \label{lemma4bis}
The action $\ell$ 
of  $\cU(\fg)$ on $L^\chi(U)$, $p^{-1}(U) \subset \Gamma$ 
leaves $\cov{\cP}$ 
invariant.
\end{proposition}

We now want to establish a fundamental pairing between
a certain Verma module and the space of polynomials inside $L^\chi(\Gamma)$.
We start by recalling the notion of pairing between $\fg$-modules.

\begin{definition} 
Let $M_1$, $M_2$ be two modules for $\fg$. By a 
\textit{$\fg$-pairing} between them we
mean a bilinear form $\langle \cdot \, , \, \cdot \rangle$
on $M_1  \times M_2$ with the property that:
$$
\langle X m_1, m_2 \rangle =  \langle m_1, -(-1)^{|m_1||X|}X m_2 \rangle, \qquad
m_i \in M_i,\, X \in \fg.
$$
\end{definition}

Since the $M_i$ are modules for $\cU(\fg)$ this implies that
$$
\begin{array}{c}
\langle X_1  \dots X_rm_1,m_2 \rangle = 
\langle m_1, (-1)^{r+|m_1|(|X_1|+ \dots + |X_r|)+l_{odd}(w)} 
X_r \dots X_1 m_2 \rangle 
\\ \\ m_i \in M_i,\, X_j \in \fg.
\end{array}
$$
where $l_{odd}(w)$ is the (minimum) number of odd transpositions appearing in 
the permutation $w:(1, \dots, r) \mapsto (r, \dots, 1)$.
The map $X \mapsto -X$ of $\fg$ is an involutive anti-automorphism of $\fg$. 
It extends uniquely to an involutive anti-automorphism 
$u \mapsto u^T$ of 
$\cU(\fg)$. 
The $\fg$-pairing requirement is equivalent to
$$
\langle u m_1,m_2 \rangle = \langle m_1, (-1)^{|u||m_1|}u^T m_2 \rangle, \qquad 
m_i \in M_i, \, u \in  \cU(\fg).
$$
We refer to this as a $\cU(\fg)$-pairing also. 
The pairing is said to be \textit{non-
singular} if 
$\langle m_1,m_2  \rangle  = 0$ for all $m_2$ (resp. for all $m_1$) 
implies that $m_1 = 0$ (resp. $m_2 = 0$).

\medskip

\begin{proposition}
Let $u$ and $v$ in $\cU(\fg)$, $f \in \cov{\cP}$.

\begin{enumerate}
\item
$
 (\partial(u^T)f)(1_G)\, = \,(\ell(u)f)(1_G)
$
\item
$
\partial(u)\ell(v)(f)(1_G)=(-1)^{|u||v|}\ell(v)\partial(u)(f)(1_G)
$
\end{enumerate}
 where $1_G$ denotes the identity element in $G$. 

\end{proposition}

\begin{proof}

$(1)$. It is enough to
prove for $u=X$ and $v=Y$ both in $\fg$. 
We can rewrite our equality as:
$$
(\ep \otimes 1)(1 \otimes -X)\mu^*(f)\, = \, 
(1 \otimes \ep)(-X \otimes 1)\mu^*(f)
$$
where $\ep$ is the counit morphism:
\(
\ep(f) = f(1_G)\quad \forall f\in  \cov{\cP}
\).
 We have:
$$
(\ep \otimes 1)(1 \otimes X)\mu^*(f)\,=\,  (1 \otimes X)(\ep \otimes 1)\mu^*(f)
\,=\,X(f)
$$
since $(\ep \otimes 1)\mu^*(f)=f$.
On the other hand:
$$
(1 \otimes \ep)(X \otimes 1)\mu^*(f)=(X \otimes 1)(1 \otimes \ep)\mu^*(f)=
X(f).
$$

\medskip\noindent
$(2)$. Again it is enough to
prove for $u=X$ and $v=Y$ both in $\fg$.
\begin{align*}
\partial(X)\ell(Y)(f)(1_G) &=
(\ep \otimes 1)(1 \otimes X)\mu^*(-Y \otimes 1)\mu^*(f)= \\ 
&=(-1)^{|X||Y|}(-Y \otimes X)\mu^*(f)(1_G).
\end{align*}
because $(\ep \otimes 1)\mu^*(f)=f$.
Similarly
\begin{align*}
\ell(Y)\partial(X)(f)(1_G)&=
(1 \otimes \ep)(-Y \otimes 1)\mu^*(1 \otimes X)\mu^*(f)=\\ 
&=
(-Y \otimes X))\mu^*(f)(1_G).
\end{align*}
\end{proof}

\begin{lemma}
\label{lemma:idealVerma}
Let $\lambda\in\fh^\ast$, and let
\[
\cM_\lambda\coloneqq \sum_{\al >0} \cU(\fg)\fg_\al + \sum_{H \in \fh}  
\cU(\fg)(H+\lambda(H) ) 
\]
then $\cM_\lambda$ is a left ideal and 
\[
\cU(\fg)=\cM_\lambda\oplus \cU(\fn^-)
\]
\end{lemma}
\begin{proof} {
For the ordinary setting this is Lemma 4.6.6 in \cite{vsv1}.
As for the supersetting it is the same. }
\end{proof}

\begin{theorem} \label{theorem5} 
There is a non-singular $\cU(\fg)$-pairing between $\cov{\cP}$ and the
Verma module $V_\lambda$. Moreover every non-zero submodule of 
$\cov{\cP}$ contains the element $\cov{1}$ 
corresponding to the constant function $1 \in \cP$. In particular,
the submodule $\cov{\cI}$ of   $\cov{\cP}$ generated 
by  $\cov{1}$  is irreducible and is the unique
irreducible submodule of  $\cov{\cP}$. 
Finally,  $\cov{\cI}$  is the unique irreducible module
of lowest weight $-\lambda$.
\end{theorem}

\begin{proof} 
The proof is the same as for the ordinary setting, let us sketch it.
We first define: 
$$
<,>:\cov{\cP} \times \cU(\fg) \lra \C, \qquad
<f,u>:=(-1)^{|u||f|}(\partial(u)f)(1_G)
$$
In order for $<,>$ to be a $\fg$-pairing, we need to verify:
$$
<\ell(c)f, u > \, = \, < f, (-1)^{|f||c|}c^Tu>, \qquad c, u \in \cU(\fg), 
f \in \cov{\cP}
$$
where $(\cdot)^T$ denotes the antiautomorphism of $\cU(\fg)$ induced by
$X \mapsto -X$ with $X\in \fg$. 
We have by the previous proposition:
$$
\begin{array}{rl}
<\ell(c)f, u > &= (-1)^{|u|(|c|+|f|)}(\partial(u)\ell(c)f)(1_G) =
(-1)^{|u||f|}(\ell(c)\partial(u)f)(1_G) =  \\ \\ 
&=(-1)^{|u||f|}(\partial(c^T)\partial(u)f)(1_G) = \\ \\
&=(-1)^{|u||f|}( \partial(c^Tu)f)(1_G) = < f,  (-1)^{|f||c|}c^Tu>
\end{array}
$$ 
By Lemma \ref{lemma:idealVerma}, in order to prove that the bilinear
 map $<,>$ descends to a $\fg$-pairing between $\cov{\cP}$ and 
$V_\lambda$ we need to prove that 
\[
u \in \cM_\lambda \quad
\Longleftrightarrow \quad (\partial(u)f)(1_{G}) =0
\]
For sufficiency, we notice that 
$<f,X_\al>= \partial(X_\al)(f)(1_{{G}})=D^L_{X_\al}(f)={\lambda(-X_\alpha)}=0$ by \eqref{def::Lchi}. 
Again by \eqref{def::Lchi}, we have that
$<f,H> = D^L_{H}(f)=-\lambda(H)f(1_{{G}})$. 
For necessity, suppose that $(\partial(u)f)(1_{{G}}) =0$ for 
each $f\in \cov{\cP}$.
By Lemma \ref{lemma:idealVerma} it is enough to notice that for 
each $X\in \cU(\fn^-)$ there exists a polynomial 
$p\in \cP$ such that $D^L_X \cov{p}(1_{{G}})\neq 0$.
So we have obtained  a nonsingular pairing 
\[
\cov{\cP}_\lambda \subset V_{-\lambda}^*
\] 
They are both weight spaces, for each weight the corresponding 
weight spaces having the
same dimension (See Corollary \ref{cor::dimspectrum}), hence they are isomorphic.

\medskip
{More explicitly,  
the functions $\cov{(t^r)}=\cov{(t_{\al_1}^{r_1} \dots  t_{\al_m}^{r_m})}$,  
corresponding to the coordinate polynomials 
$t^r$ on $N^-$, are weight vectors for the action of $\fh$ for the
weight $r-\lambda$. Hence $\cov{\cP}$ is a weight module with the multiplicities 
defined in Sec. \ref{actiontorus-sec}. 
We shall prove that every non-zero $\ell$-invariant subspace W
of $\cov{\cP}$ contains the vector $\cov{1}$ 
defined by the constant function $1$ on $N^-$.
Now $W$ is a sum of weight spaces and if it does not contain $\cov{1}$, then 
$W$ is contained in the sum of all weight spaces corresponding to the weights
$-\lambda + r$ where $r = (r_i)$ with some $r_i > 0$. Now 
$<\ell(H)m_1,m_2> = 
-<m_1,\partial(H)m_2>$, for all $H \in \fh$, $m_1 \in \cov{\cP}$, 
$m_2 \in V_\lambda$. 
This shows that the weight space of $\cov{\cP}$
for the weight $\theta$ is orthogonal to the weight space of $V_\lambda$ 
for the weight $\phi$ unless $\theta = -\phi$.
Let $v$ be a non-zero vector of highest weight $\lambda$ in $V_\lambda$. Since
$W$ is contained in the span of weights other than $-\lambda$, 
we have $<W, v> = 0$.
Hence, for all $g \in \cU(\fg)$, $w \in W$ we have 
$<\ell(g)w, v> = 0$. So $<w, g^T v> = 0$
for all $g \in \cU(\fg)$. But $v$ is cyclic for 
$V_\lambda$ and so we have $<w, V_\lambda> = 0$ for
all $w \in W$. This means that $W = 0$, contradicting the hypothesis that
$W \neq 0$.
Thus every non-zero submodule of $\cov{\cP}$ contains the submodule $\cov{\cI}$
generated by $\cov{1}$. This submodule is then the unique irreducible 
submodule of $\cov{\cP}$.
The weights of $\cov{\cI}$ are of the form $-\lambda+d$ where 
$d$ is a positive integral
linear combination of the simple roots, and $\cov{1}$ has weight $-\lambda$. 
It is then
clear that $\cov{1}$ is the lowest weight of $\cov{\cI}$. 
This fact, together with its
irreducibility, characterizes it uniquely.
}

\end{proof}

\subsection{Harish-Chandra decomposition} \label{hcdec-sec}

From now on we fix a positive
admissible system $P$ for $\fg$ (see Sec. \ref{admissible-sec}).

Let $P^+$ be the analytic supergroup corresponding to the
subsuperalgebra $\fp^+=\sum_{\be\in P} \fg_\be$. 
Similarly define $P^-$. Let $K$ be the (ordinary) 
analytic subgroup of $G$ corresponding to the
Lie superalgebra $\fk=\fk_0$.

\begin{proposition} \label{lemma7}
The morphism $\phi:P^- \times K \times P^+ \lra G$, defined as
$(p^-,k,p^+) \mapsto p^-kp^+$ in the functor of points notation,
is a complex analytic isomorphism of $P^- \times K \times P^+$ onto 
an open set $\Omega \subset G$.
\end{proposition}

\begin{proposition}\label{prop-lemma2}
\begin{enumerate}
\item $G_rK^\R(P^+)^\R$ is an open subsupermanifold in $(P^-KP^+)^\R$.
\item $G_rKP^+=(\wt{G_r}\wt{K}\wt{P^+},
\cO_{P^-KP^+}|_{\wt{G_r}\wt{K}\wt{P^+}})$ is a complex open subsupermanifold in $P^-KP^+$.
\item $\wt{G_r} \cap \wt{K}\wt{P^+}=\wt{K_r}$.
\end{enumerate}
\end{proposition}

\begin{proof} For the first statement,
observe that $\fg_r \oplus \fk \oplus \fp^+=\fg$, by Prop. 
\ref{lemma3}, because 
$\fk \oplus \fp^+ \supset \fb$.
Hence by Lemma \ref{lemma1} we have that $G_rKP^+$ is open in $G$ and
since $\wt{G_r}\wt{K}\wt{P^+} \subset \wt{P^-}\wt{K}\wt{P^+}$ (see \cite{helgason} pg. 389), 
we have
$G_rKP^+$ is open in $P^-KP^+$.
The second statement is topological, so it is
true because of the ordinary theory. 
\end{proof}

We now turn to the construction of the complex structure of $G_r/K_r$.

\begin{proposition}  \label{cpx-st} We have
$G_r/K_r \cong (G_{{r}}KP^+/KP^+)^\R$. Hence 
and $G_r/K_r$ acquires a natural 
{$\wt{G_r}$ invariant complex structure}. 
\end{proposition}

\begin{proof}
This is an immediate consequence of Lemma \ref{lemma1} (2) and (3), together
with Prop. \ref{prop-lemma2}. 
\end{proof}

\subsection{Harish-Chandra representations and their geometric
realization}

\label{actionGr-subsec}

In this section we give a global realization 
of the Harish-Chandra infinitesimal representations
studied in Sec. \ref{hcmod-sec}. 

\begin{definition}
Let the complex vector superspace $V$ be a $\fg$-module via the
representation $\pi$.
We say that $V$ is a \textit{$(\fg_r, K_r)$-module}
if there exists a representation $\pi_{K_r}$ of $K_r$ such that
\begin{enumerate}
\item $\pi(\Ad(k)X)=\pi_{K_r}(k)\pi(X)\pi_{K_r}(k)^{-1}$
\item $V=\sum_\tau V(\tau)$ where the sum is algebraic 
and direct and $V(\tau)$
is the span of all the linear finite dimensional subspaces 
corresponding to the irreducible representation associated with
the $K_r$-character $\tau$.
\end{enumerate}
We say that $V$ is a \textit{$(\fg_r, \fk_r)$-module} if 
$$
V=\sum_{\theta \in \Theta} V(\theta)
$$ 
where the sum is algebraic and direct, $\Theta$ denotes the
set of equivalence classes of the finite dimensional irreducible
representations of $\fk$ and $V(\theta)$ is the sum of all
representation occurring in $V$ lying in one of such classes
$\theta \in \Theta$.
\end{definition}

We say that the  $(\fg_r, K_r)$-module $V$ is a \textit{Harish-Chandra
module} (or HC-module for short) 
if each $V(\tau)$ is finite dimensional and 
$V$ is finitely generated as $\cU(\fg)$ module. 
Similarly we can define
also the notion of Harish-Chandra modules for $(\fg_r, \fk_r)$-modules.

\medskip
We say that a vector is \textit{$K_r$-finite} if it lies in a 
finite dimensional
$K_r$ stable subspace.

We now want to study the action of $G_r$ on a superspace of sections
of the line bundle $L^\chi$ over $\red{G/B^+}$. 
Since $\red{G_rB^+}$ is open in $\red{G}$
(see Lemma \ref{lemma2}), we can consider
$L^\chi(G_rB^+)$ and
since $G_r$ acts on the left
on $G_rB^+$ we have a well defined action of $G_r$ on 
the  Frech\'et superspace $L^\chi(G_rB^+/B^+)$:
$$
\begin{cases}
(g \cdot f) =l_{g^{-1}}^*f & g\in \red{G_r}\\
X.f  = D^R_{-X} f & X\in \fg
\end{cases}
$$
Next lemma is a simple generalization of  Theorem 11, 
pg 312 in \cite{vsv3} {and holds in a general setting}.

\begin{lemma}
\label{lemma:Raja}
Let the notation be as above. 
Let $F$ be a 
 Frech\'et representation of $G_r$ on which $G_r$ acts
via $\pi=(\pi_0, \rho)$.  If $v$ is a 
weakly analytic vector for $\red{G_r}$, then 
\[
\overline{\cU(\fg)v}\subseteq F
\]  
is the smallest closed $G_r$-invariant subspace of $F$ containing $v$.
\end{lemma}
\begin{proof}
For each $X\in \cU(\fg)$ and $\lambda \in F^\ast$ 
(the topological dual of $F$), define the function 
$f_{X,\lambda}\colon \red{G_r}\to \C$:
\[
f_{X,\lambda}(g)=\lambda(\pi_0(g){\rho(X)}v)
\]
Let $\lambda$ be such that $\lambda=0$ on $\overline{\cU(\fg)v}$. 
It is easily checked that the infinitesimal action of $\cU(\fg_{{0},r})$ 
preserves the analytic vectors hence 
we obtain
\[
Zf_{X,\lambda}({1_G})=0 \quad \mbox{ for each } Z\in \cU({\fg_0})
\]
Since $\red{G_r}$ is connected, we conclude that $f_{X,\lambda}=0$ on 
$\red{G_r}$.
Hence by the Hahn-Banach theorem we conclude that 
$\mathrm{span}\{\pi_0(G){\rho(X)}v\}$ is contained in 
$\overline{\cU(\fg)v}$. Since this is true for all 
$X\in \cU(\fg)$ we conclude that
$\overline{\cU(\fg)v}$ is $\red{G_r}$ invariant. Since
\[
{\rho}(Y)\pi_0(g){\rho(X)} v= \pi_0(g) {\rho((g^{-1}Y) X)} v
\] 
it is also clear that it is the smallest $G_r$-invariant subspace 
of $F$ containing $v$.
\end{proof}

Let us introduce the notation:

\beq 
\wt{L^\chi(U)}:=\{\wt{f} \, | \, f \in L^\chi(U)\}\label{red-space}
\eeq

\noindent 
which is meaningful because \(L^\chi\) is a subsheaf of the structural sheaf.

\begin{theorem} \label{theorem6} 
Let $S=G_rB^+/B^+$ and assume $\wt{F^1}:=\wt{L^\chi(S)} \neq 0$.
Then:
\begin{enumerate}
\item $F^1$ contains an element $\psi$ which is an
analytic continuation of $\cov{1}$; 
\item $F^{11}:= \overline{\ell(\cU(\fg))\psi}$ $\subset  L^\chi(S)$
is a Fr\'echet $G_r$-module, $K_r$-finite and with $K_r$-finite
part  $\ell(\cU(\fg))\psi) = \cov{\cP}_\lambda$.
\item 
the $K_r$-finite
part  ${\cov{\cP}_{\lambda}}$ is
isomorphic to $\pi_{-\lambda}$ the irreducible representation
with lowest weight $-\lambda$. In particular 
$\lambda(H_\al) 
\in \Z_{\geq 0}$ for all compact positive roots $\alpha$.
\end{enumerate}
\end{theorem}

\begin{proof} 
We first establish the $K_r$-finitess of $F^1$.
From (2) of Corollary \ref{spectrum-cor2} it follows  that the 
subspace $F^1(\tau)$ injects (through the restriction morphism) 
in $F^2(\tau)$, where $F^2$ denotes $L^\chi(\Gamma_2)$ 
(we recall that $\Gamma_2=(G_r B^+\cap \Gamma)^0/B^+$). 
From (3) 
of Corollary \ref{spectrum-cor2}, we know that $\dim F^2(\tau)=\dim F(\tau)$. 
By (3) of Corollary \ref{spectrum-cor} and (2) of Corollary 
\ref{cor::dimspectrum}, we finally obtain $\dim F(\tau)<+\infty$. 
Hence  $F^1$ is $A_r$-finite. By Corollary \ref{finiteness-cor}
the $A_r$ finiteness implies in our case the $K_r$-finiteness. 
Hence $F^1$ is $K_r$-finite

\medskip
We now go to the proof of $(1)$.
Assume that the $K_r$-finite part  $(F^1)^0=\sum F^1(\tau)$ does not 
include the weight $-\lambda$, in other words we assume there is no analytic
continuation of $\cov{1}$ to $S$. 
By Corollary \ref{spectrum-cor} and Corollary \ref{spectrum-cor2}, $(F^1)^0$ is isomorphic  to a subset of the set of polynomials in the $t_\alpha$ (see Sec.
\ref{linebu-sec} for the notation). Since   $\cov{1}\in F^2(-\lambda)\supseteq F^1 (-\lambda)$ we have that all the elements $f$ in $(F^1)^0$ are zero when evaluated at $1_G$. Hence, by the density of $(F^1)^0$ in $F^1$, all the elements in $F^1$ vanish at $1_G$. Using the $G_r$ action it follows that $\tilde{f}=0$  for all  $f\in F^1$.

\medskip
As for $(2)$, $F^{11}:= \overline{\ell(\cU(\fg))\psi}$ is a  Frech\'et 
superspace, since it is a closed  subspace of a  Frech\'et superspace. 
The fact that $F^{11}$ is a $G_r$-module follows from Lemma~\ref{lemma:Raja}. 
Hence $F^{11}$ is a $G_r$ submodule of 
$F^1$, and  it is $K_r$-finite since it is a submodule of
the $K_r$ finite module $F^1$. $\cJ=\ell(\cU(\fg))\psi$ is clearly 
a highest weight $\cU(\fg)$ module. 
Since $F^{11}$ is the closure of the $K_r$-finite
subspace {$\cJ$} subspace, its $K_r$-finite part is
precisely $\cJ$.

\medskip

$(3)$. We know that $\cJ \subset (F^1)^0 \hookrightarrow (F^2)^0\simeq F^0$.
Clearly $\cJ \hookrightarrow \cov{\cI}:=\cU(\fg)\cov{1} \subset F^0$,
but since by \ref{theorem5} $\cov{\cI}$ is irreducible, we have 
$\cJ=\cov{\cI}$. $\cJ$ is the irreducible lowest weight module of
lowest weight $-\lambda$ or equivalently $\cJ$ is the irreducible 
highest weight module of highest weight $-\lambda$ with respect to
the positive system $-P$. The $K_r$-finiteness of $\cJ$ implies that
$-\lambda(H_{-\alpha}) \geq 0$, hence our result.

\end{proof}

\begin{corollary}\label{cor-thm6}
Let the notation be as above. Then $\cJ=\cU(\fg)\psi \subset (F^1)^0$
is the irreducible 
Harish-Chandra module with highest weight $-\lambda$
with respect to the positive system $-P$.
\end{corollary}

\begin{proof} This is an immediate consequence of Proposition
\ref{infinitesimalhc}. 
\end{proof}

\begin{definition}
We say that a dominant integral 
weight $\lambda$ is $K$\textit{-integrable} if
the {$\fk$} irreducible representation associated with $\lambda$ 
can be lifted to {$K$}. 
\end{definition}

\begin{definition}
An holomorphic character  $\chi_\lambda$ of \(A\)  is $K$\textit{-integrable} if \(\lambda\)  is dominant integral for the positive compact roots and if the associated {$\fk$}  representation  can be lifted to {$K$}. In this case, we also say that \(\lambda\) is \(K\)-integrable.
\end{definition}

As in the classical case we have the following Lemma.

\begin{lemma}
If \(\lambda\) is \(K\)\textit{-integrable} then the associated representation of \(K\) is  finite-dimensional and  holomorphic.
\end{lemma}
\begin{proof}
This fact is entirely classical and it is proved in \cite{hc}.
\end{proof}

\begin{theorem} \label{theorem8}
Let the notation be as above. Assume the following:
\begin{itemize}
\item $\dim(\fc)\geq 1$. 
\item $\lambda \in \fh^*$ is integral and
$\lambda(H_\al) \geq 0$ for all $\al$ compact positive root. 
\item $\lambda$ is  $K$-integrable. 
\end{itemize}
Then 
$\wt{L^\chi(G_rB^+)} \neq 0$ and
$F^{11}$ is a $G_r$ representation
whose $K_r$-finite part is the lowest weight representation
$\pi_{-\lambda}$. 
\end{theorem}

\begin{proof}
It is enough
to show that $\wt{L^\chi(G_rB^+)} \neq 0$, since (2) is
an immediate consequence of Theorem \ref{theorem6}.
Let $\sigma_\lambda$ be the finite dimensional irreducible representation 
of $K$ with highest weight $\lambda$ on the  vector space $V$.
Let $v_\lambda$ be the corresponding highest weight vector. 
We can define the coefficient of the representation $\sigma_\lambda$ 
corresponding to  $v_\lambda$ that is the nonzero section  
$a_{11}:K \lra \C$, $a_{11}(k)=
(\sigma_\lambda(k)v_\lambda)_{v_\lambda}$, that is the 
$v_\lambda$  component of $\sigma_\lambda(k)v_\lambda$ 
corresponding to the weight decomposition of $V$. 
Using {Prop.}~\ref{lemma7}, we can extend $a_{11}$  to a nonzero section 
in $\cO(P^-KP^+)$. Since $G_r$ is embedded into 
$(P^-KP^+)^\R$ and $a_{11}(1_G)=1$
we obtain a non zero section of $G_r$, that is $a_{11} \in \cO(G_r)$. 
It is immediate to verify:
$$
r_b^\ast a_{11}=\chi_0^\lambda(b)^{-1}a_{11}, \quad b \in \widetilde{B^+},
\qquad D^L_{X} a_{11}=-\lambda(X) a_{11}, \quad X \in \fb^+
$$
so that $a_{11} \in L^\chi(S)$ as requested.
\end{proof}

\subsection{The Siegel superspace}\label{siegel-subsec}
To illustrate  the theory developed so far, we want to give an example, interesting by itself, where the various geometrical tools developed so far (admissible systems, symmetric superspaces,...) come into play.

Consider the closed {complex}
analytic subsupergroup $P$ of the {complex} orthosymplectic supergroup 
$\rOsp(m|2n)$ defined, 
via its functor of points, as: 
$$
P(T)=\left\{\left( \begin{matrix} a & 0 & \al_2 \\ 
b_{11}\al_2^t a & b_{11} & b_{12} \\
0 & 0 & (b_{11}^t)^{-1}\end{matrix} \right) \mid 
\begin{cases}
a^t a= 1\\
(b_{11}^{-1}b_{12})-(b_{11}^{-1}b_{12})^t=\al_2^t\al_2
\end{cases}
\right\}
$$
where $a \in \rGL(m)(T)$, $b_{11} \in \rGL(n)(T)$, 
$b_{12} \in \rM(n)(T)$, $\al_2 \in \rM(m|0,0|n)(T)$, $T \in \smflds_\C$,
($\rGL$, $\rM$ denoting respectively the general linear (super)group
and the (super)matrices).

\medskip

As for any closed analytic subsupergroup of an analytic Lie supergroup
it is possible to construct the quotient $\rOsp(m|2n)/P$. This is a complex
analytic supermanifold, that we call the \textit{super Lagrangian} 
and denote it by $\cL$. 
Notice that, 
by the very definition, the reduced manifold of \(\cL\) is $\wt{\cL}$,
the ordinary Lagrangian manifold in $\rSp(2n,\C)$,
and  we have a natural transitive action of 
$\rOsp(m|2n)$ on  the supermanifold $\cL$.
We also define $\cLf$ as the open subsupermanifold of $\cL$ corresponding to
the open subset $\cLfo$ of  $\wt{\cL}$:
$$
\begin{array}{rl}
\cLfo&=\left\{\left( \begin{array}{c} Z_1 \\ Z_2  \end{array} \right) \,
| \,  Z_1^tZ_2 \, \, \hbox{symmetric}, \,
\det(Z_2) \neq 0 \right\}\Big/\rGL(n)(\C)  \\ \\
&\cong 
\left\{  
\left(  \begin{array}{c} Z  \\ 1  \end{array} \right)  \, \big| \,
Z \, \hbox{symmetric} 
\right\} =  \{  X+iY \, | \, X, Y \in \rM(2n, n)(\R), \,
\hbox{symmetric}\}.
\end{array}
$$

We now want to characterize the functor of points of $\cLf$. We start
observing that
we can always choose {\sl uniquely} a representative of the class 

\begin{equation}
\left( \begin{matrix} a & \al_1 & \al_2 \\
\be_1 & b_{11} & b_{12} \\ \be_2  & b_{21} & b_{22}
\end{matrix} \right)P(T)\in \cLf, \quad
\hbox{in the form} \quad
\label{elform}
\left( \begin{matrix} 1 & \zeta & 0 \\
\zeta^t & z & -1 \\ 0 & 1 & 0\end{matrix} \right)
\end{equation}

This is equivalent, to 
to find
 $z$, $\zeta$, $u$, $v$, $w$, $\xi$ depending on $\al_i$, $\be_i$
and $a$, $b_{ij}$ in the following equation:
$$
\left( \begin{matrix} 1 & 0 & -\zeta \\
0 & 0 & 1 \\ \zeta^t & -1 & z-\zeta^t \zeta\end{matrix} \right)
\left( \begin{matrix} a & \al_1 & \al_2 \\
\be_1 & b_{11} & b_{12} \\ \be_2  & b_{21} & b_{22}
\end{matrix} \right)=
\left( \begin{matrix} u & 0 & \xi \\
v\xi^tu & v & w \\ 0 & 0 & (v^t)^{-1}\end{matrix} \right)
$$
Notice that there is no loss of generality in assuming $b_{21}=1$.
The check the solutions are unique and
compatible with the conditions defining $\rOsp(m|2n)$ is a direct calculation.
The values obtained are:
$$
u=a-\al_1\be_2, \quad \xi=\al_2-\al_1b_{22}, \quad v=1, 
\quad w=b_{22}, \quad z=b_{11}, \quad \zeta=\al_1
$$

\begin{proposition} \label{tpts}
The $T$-points of the supermanifold $\cLf$ are identified
with the matrices in $\rOsp(m|n)(T)$ of the form:
$$
\cLf(T)\simeq \left\{\left( \begin{array}{ccc} 1 & \zeta & 0 \\
\zeta^t & z & -1 \\ 0 & 1 & 0\end{array} \right)\, | 
\,\zeta^t \zeta + z^t -z=0 
\right\}.  \quad 
\hbox{Hence} \quad
\cLf \cong \C^{\frac{n^2+n}{2}|mn}.
$$
\end{proposition}

\begin{proof} 
Let us choose a suitable open cover $\{T_i\}_{i \in I}$ of $T$, so that
$$\cL (T_i)=(\rOsp(m|2n)/P)(T_i)=\rOsp(m|2n)(T_i)/P(T_i).$$
We can then  write
$$
\cLf(T_i)=\left\{
\left( \begin{array}{ccc} a & \al_1 & \al_2 \\
\be_1 & b_{11} & b_{12} \\ \be_2  & b_{21} & b_{22}
\end{array} \right)P(T_i)\, | \, b_{21} \,
\hbox{invertible} \, \right\}\, 
$$
By (\ref{elform}), we can write:

\begin{equation}\label{siegel-eq}
\left( \begin{array}{ccc} a & \al_1 & \al_2 \\
\be_1 & b_{11} & b_{12} \\ \be_2  & b_{21} & b_{22}
\end{array} \right)P(T_i) \, = \,
\left( \begin{array}{ccc} 1 & \zeta & 0 \\
\zeta^t & z & -1 \\ 0 & 1 & 0\end{array} \right)P(T_i), 
\qquad  \zeta^t \zeta + z^t -z=0
\end{equation}
$\cLf$ is then defined by $n(n-1)/2$ equations in $\C^{n^2|mn}$:
$$
\sum_k \zeta_{ki}\zeta_{kj} + z_{ji}-z_{ij}=0, \qquad 1 \leq i < j \leq n
$$
\end{proof}

\begin{definition}
We define \textit{Siegel superspace} the open supermanifold of $\cLf$
corresponding to the complex open subset: 
$$
\wt{\cS}=\left\{Z \,|\, Z=X+iY, \, X, Y \in \rM(2n,n)(\R) \,
\hbox{symmetric}, \, Y>0 \right\}
\subset \cLfo$$
\end{definition}

By the Chart Theorem (see Ch. 4 in \cite{ccf}) we have that
a $T$-point of the Siegel superspace
$\wt{\cS}$ corresponds to a choice of two matrices $\zeta$ and $z$ with
entries in $\cO(T)$ such that their values at all topological
points of $\wt{T}$ land in $\wt{\cS}$. In other words, $\cS(T)$ consists
of the following elements in 
$\cLf(T)$: 
\begin{align*}
\cS(T) & = \left\{ 
\left( \begin{matrix} 1 & \zeta & 0 \\
\zeta^t & z & -1 \\ 0 & 1 & 0\end{matrix} \right)\mid \begin{cases}
\zeta^t \zeta + z^t -z=0 \\
 z=x+iy, \, 
\widetilde{y}(t)>0, \, \forall t \in \wt{T} 
\end{cases}
\right\} \subset \cLf(T)\\
\end{align*}
We now want to realize the Siegel superspace as a real homogeneous
supermanifold (for our notation
see \cite{cfk1}). 

\medskip
Consider the natural action of the real
orthosymplectic supergroup $\rOsp(m|2n, \R)$ on the quotient
$\rOsp(m|2n)/P$ and restrict it to $\cS$:
\begin{align*}
{\rOsp(m|2n,\R)}(T) \times \cS(T) & \lra   \cS(T) \\ 
\left( \begin{matrix} a & \al_1 & \al_2 \\
\be_1 & b_{11} & b_{12} \\ \be_2  & b_{21} & b_{22}
\end{matrix} \right) , \,\left( \begin{matrix} 
1 & \zeta & 0 \\
-\zeta^t & z & -1 \\
 0 & 1 & 0
 \end{matrix} \right) & \mapsto 
\left( \begin{matrix}
 1 & 
{a\zeta + \al_1z+\al_2  \over \be_2\zeta+b_{21}z+b_{22}} & 0 \\ 
\left({a\zeta + \al_1z+\al_2  \over \be_2\zeta+b_{21}z+b_{22}}\right)^t & 
{\be_1\zeta+b_{11}z+b_{12} \over \be_2\zeta+b_{21}z+b_{22}} 
& -1 \\ 
0 & 1 & 0
\end{matrix} \right)
\end{align*}

\begin{theorem} \label{ospaction-thm}
$\rOsp(m|2n,\R)$ acts transitively on the Siegel superspace and
the stabilizer of the topological point $(iI,0) \in \wt{\cS}$ is the subgroup:
$$
K_\br(T)=Stab(iI,0)(T)= \left\{
\left( \begin{matrix} a & 0 & 0 \\
0&b_{11}& b_{12}  \\ 
0&-b_{12}  & b_{11}  
\end{matrix} \right) \in \rOsp(m|2n,\R)(T) \, \right\}, \qquad T \in \smflds_\R 
$$
which is compact and coincides with its reduced group: $(K_\br)_{red}=K_\br$
and it is equal to $\mathrm O(m) \times  \mathrm U(n)$.
So we have the isomorphism as real supermanifolds:
$$
\cS \cong \rOsp(m|2n,\R)/K_\br.
$$
\end{theorem}

\begin{proof}
The action of $\wt{\rOsp(m|2n, \R)}$ on $\wt{\cS}$ is transitive. 
Consider the supermanifold morphism
$a_p: \rOsp(m|2n, \R) \lra \cS$, $a_p(g)=g \cdot (0,iI)$. 
The differential $(da_p)_I$ at the identity is surjective, hence
the result follows (see also Prop. 9.1.4 in \cite{ccf}). 
\end{proof}

The form we have found for $K_\br$ is not suitable for Lie superalgebra
calculations, so we need to transform $\cS$, so that also $K_\br$ 
transforms accordingly. 
We shall do this via the
{\sl super Cayley transform}.

\medskip

Consider the following linear transformation:
$$
L=\left( \begin{matrix} 1 & 0 & 0 \\
0& i/\sqrt{2i}&  i/\sqrt{2i} \\ 
0&-1/\sqrt{2i}  & 1/\sqrt{2i}  
\end{matrix} \right) \in  \wt{\rOsp(m|2n)} 
$$
where we write $1$ in place of the identity matrix.

Define the open subsupermanifold $\cD=(\wt{\cD}, \cO_{\cLf|_{\wt{\cD}}})$
of $\cLf$ with topological space:
$$
\wt{\cD}=\left\{\begin{pmatrix}0 \\ z \\ 1 \\ \end{pmatrix}
\, |\, z \in \rM(n , n)(\C) \, \hbox{symmetric}, \, 
1-z\zbar>0 \right\}
$$ 

\begin{proposition} 
The linear transformation $L$ 
induces a supermanifold diffeomorphism:
\begin{align*}
\phi_T: \cD(T) & \lra  \cS(T) \\
 \left(\begin{matrix}\eta \\ z \\ 1 \\ \end{matrix} \right)
& \mapsto   \begin{pmatrix} 
\sqrt{2i}\eta(1-z)^{-1} \\ i(z+1)(1-z)^{-1} \\ 1 \end{pmatrix}
\end{align*}
\end{proposition}

\begin{proof} Let us take a generic element in 
$\cLf(T)$ and multiply it by $L$:
$$
\left( \begin{matrix} 
1 & 0 & 0 \\
0& i/\sqrt{2i}&  i/\sqrt{2i} \\ 
0&-1/\sqrt{2i}  & 1/\sqrt{2i}  
\end{matrix} \right) \left(\begin{matrix}\eta \\ 
z \\ 1 \\ \end{matrix}\right)=
 \left(\begin{matrix}\eta \\ \frac{i}{\sqrt{2i}}(z+1) \\ 
\frac{1}{\sqrt{2i}}(1-z) \\ \end{matrix}\right)\sim
\left(\begin{matrix} \sqrt{2i}\eta(1-z)^{-1}
\\i(z+1)(1-z)^{-1}\\1 \end{matrix}\right)
$$
The map $\phi$ can be extended on the lagrangian $\cLf$
(except at the locus $z=1$) and it is differentiable on 
$\wt{\cLf}\setminus \{z=1\}$. Since $\wt{\phi}$ is an homeomorphism
when restricted to $\wt{\cD}$ and  the  differential 
$\mathrm{d} \phi$ is surjective, the result follows (see \cite{ccf}). 
\end{proof}

We call the diffeomorphism $\phi$ the \textit{super Cayley transform}.
Define:
$$
\rOsp_\DD(m|2n)(T) := L^{-1} \rOsp(m|2n,\R)(T) L, \qquad
K_\DD(T):= L^{-1} K_\br(T) L
$$

\begin{proposition}
$\rOsp_\DD(m|2n)$ is a real form of the orthosymplectic supergroup and
its functor of points is explicitly given by: 
$$
\rOsp_\DD(m|2n)(T) = 
\left\{ \left( \begin{matrix} a_0 & \al_1 & -i\albar_1 \\
\be_1 & b_{11} &  b_{12} \\ 
i\bebar_1& \overline{b_{12}}  & \overline{b_{11}}  
\end{matrix} \right) \right\} \subset \rOsp(m|2n)^\R(T)
$$
where $a_0 \in O(m)$, $T$ is a real supermanifold and,
as usual, $\rOsp(m|2n)^\R$ is the complex orthosymplectic supergroup
viewed as a real supergroup.

\medskip

$K_\DD$ is a real form of the
compact group $K_\br$ and it is given by 
$$
K_\DD= 
\left\{ \left( \begin{matrix} a_0 & 0 & 0 \\
0 & b_{11} &  0 \\ 
0 & 0  & \overline{b}_{11}  
\end{matrix} \right) \, \Big| \, b_{11}\overline{b}_{11}^t=1, \,
a_0 \in O(m) \right\}
$$
\end{proposition}

\begin{proof} 
The conjugation defining $\rOsp_\DD(m|2n)$ inside $\rOsp(m|2n)$ is:
$$
\left( \begin{matrix} a & \al_1 & \al_2\\
\be_1 & b_{11} &  b_{12} \\ 
\be_2& b_{21}  & b_{22}  
\end{matrix} \right) \mapsto 
\left( \begin{matrix} \bar{a} & -i\albar_2 & -i\albar_1 \\
i\bebar_2 & \overline{b_{22}} & \overline{b_{21}} \\ 
i\bebar_1& \overline{b_{12}}  &\overline{b_{11}}  
\end{matrix} \right)
$$
The statement about $K_\DD$ is entirely classical and known.
\end{proof}

\begin{proposition}
$\rOsp_\DD(m|2n)$ acts transitively on $\cD$ and $K_\DD$ is the stabilizer
of the topological point $(1,0)$. Hence
$$
\cD \cong \rOsp_\DD(m|2n)/K_\DD.
$$
\end{proposition}

\begin{proof} It the same as for \ref{ospaction-thm}.
\end{proof}

We now compute the real Lie superalgebras of $\rOsp_\DD(m|2n)$ and $K_\DD$.

\begin{proposition} \label{cartan-decthm}
We have that 
$$
\begin{array}{rl}
\rosp_\DD(m|2n) 
&=\left\{ \left( \begin{matrix} x & \xi & -i\overline{\xi} \\
-i\overline{\xi^t} & y_{11} &  y_{12} \\ 
-{\xi^t} & \overline{y_{12}}  & \overline{y_{11}}  
\end{matrix}\right) \, | \, 
\begin{cases} y_{12} \, \mbox{ \textrm{symmetric}}, \\
x=-\overline{x}^t, \, y_{11}=-\overline{y_{11}^t} 
\end{cases}
\right\} \, \\ \\
\fk_\DD &=\Lie(K_\DD)= 
\left\{ \left( \begin{matrix} x & 0 & 0 \\
0 & y_{11} &  0 \\ 
0 & 0  & \overline{y_{11}}  
\end{matrix}\right)\right\}
\end{array}
$$
\end{proposition}
\begin{proof}
The conjugation defining $\rosp_\DD(m|2n)$ is obtained as follows:\\
$X \in \rosp_\DD(m|2n)$ if and only if $X \in \rosp(m|2n)$ and
$F \overline{X}=XF$ where
$F= \left( \begin{matrix} 1 & 0 & 0 \\ 0 & 0 & i \\
0 & i & 0 \end{matrix} \right)$.
An easy calculation shows the result.
\end{proof}

For the complex Lie superalgebra (see Sec. \ref{admissible-sec})
$\rosp(m|2n)$ we have the admissible system $P=P_k \cup 
P_{n,0} \cup P_{n,1}$, where:
\begin{align*}
P_{n,0} = & \{\ep_1 \pm \ep_j\mid 1<j\leq m  \}\cup\{\ep_1\} 
\cup \{\de_i+\de_j\mid 1\leq i,j \leq n \} \\ \\ 
P_{n,1} = & \{\de_i \pm \ep_j\mid 1\leq i\leq n \,,\, 
1\leq j\leq m\}\cup \{\delta_i\mid 1\leq i \leq n\}\\ \\
P_{k}= &
\{\ep_i \pm \ep_j\mid 1<i<j\leq m \}\cup 
\{\ep_i\mid 1< i \leq m \} \cup\{ \de_i - \de_j\mid1\leq i<j\leq n\}
\end{align*}
with 
$$\Pi=\{ \ep_1 - \ep_2\,, \dots ,\, \ep_{m-1}-\ep_m\,,\, \ep_m\,,\, 
\de_1 - \de_2\,, \dots ,\, \de_{n-1}-\de_n\,,\, \de_{n}-\ep_1\}
$$
the simple system with 
one simple non-compact even root $\ep_1 - \ep_2$ 
and one non-compact simple odd root: $\de_{n}-\ep_1$.

\begin{proposition}\label{liestructure}
The complex Lie superalgebra $\rosp(m|2n)$ is the vector space
direct sum 
of three Lie subsuperalgebras:
$$
\rosp(m|2n)=\fk\oplus \fp^+ \oplus \fp^-, 
$$
$$
\fk=\sum_{\al \in P_k\cup -P_k} \fg_\al, \quad
\fp^+=\sum_{\al \in P_n} \fg_\al, \quad
\fp^-=\sum_{\al \in -P_n} \fg_\al.
$$
where $\fk=\C \otimes \Lie(K_\DD)$ and 
$$
\fp^+=\left\{ \left( \begin{array}{ccc} 
0 & 0 & \xi \\ \xi^t & 0 & u \\ 0 & 0 & 0
\end{array} \right) \right\} \quad 
\fp^-=\left\{ \left( \begin{array}{ccc} 
0 & -\eta & 0 \\ 0 & 0 & 0 \\ \eta^t & v & 0
\end{array} \right) \right\}. 
$$
\end{proposition}

We can now express explicitly the Harish-Chandra decomposition for
$\rOsp(m|2n)$, proven in Prop. \ref{lemma7}. 

\medskip
Let $P^-$ and $P^+$ be the complex subsupergroups of the
complex orthosymplectic supergroup $\rOsp(m|2n)$ defined via
their functor of points as:
$$
P^+=\left\{ \left( \begin{array}{ccc} 
1 & 0 & \xi \\ \xi^t & 1 & u \\ 0 & 0 & 1
\end{array} \right) \right\} \quad 
P^-=\left\{ \left( \begin{array}{ccc} 
1 & -\eta & 0 \\ 0 & 1 & 0 \\ \eta^t & v & 1
\end{array} \right) \right\}
$$
Most immediately $\fp^\pm=\Lie(P^\pm)$. 
Notice that while in the ordinary setting we have that the groups 
$\wt{P}^\pm$ are
abelian, in the supersetting, this is no longer true.

By Prop. \ref{lemma7} we have that
the supermanifold $P^-KP^+$ is open in $\rOsp(m|2n)$.

\medskip
By its very construction $\cD$ is
a complex supermanifold and it has a natural
action of $\rOsp_\cD(m|2n)$.
Notice that
$$
J \circ \ad(X)|_{\fp_\cD}=  \ad(X)|_{\fp_\cD} \circ J 
$$
with $J$ the almost complex structure at the identity coset,
$J: \fp_\cD \lra \fp_\cD$, where we identify
$\fp_\cD=T_{K_\cD}(\rOsp_\cD(m|2n)/K_\cD)$. 

\begin{proposition} \label{complex-str-prop}
Let $J=\ad(c)|_{\rosp_\cD(m|n)_0}+
\ad(2c)|_{\rosp_\cD(m|n)_1}$, where $c$ is the element in the
center of $\fk_\cD$:
$$
c=\begin{pmatrix} 0 & 0 & 0 \\ 0 & i/2 & 0 \\ 0 & 0 & -i/2 \end{pmatrix}
$$
Then $J|_{\fp_\cD}$ defines a complex structure on $\fp_\cD$,
{which corresponds to the $\wt{\rOsp_\cD(m|2n)}$ invariant
complex structure on  $\rOsp_\cD(m|2n)/K_\cD$},
as in Prop.  \ref{cpx-st}. 
\end{proposition}

\end{document}